\documentclass[a4paper, 10pt]{amsart}

\usepackage{xfrac}

\usepackage{amsmath,amssymb,amsthm}
\usepackage{enumitem}
\usepackage{hyperref}
\usepackage{listings}
\usepackage{comment}

\usepackage[nameinlink, capitalize]{cleveref}
\usepackage{etexcmds}
\usepackage{thmtools}
\usepackage{mathtools}

\newcommand{\eps}{\varepsilon}
\newcommand{\R}{\mathbf{R}}
\newcommand{\C}{\mathbf{C}}

\newcommand{\LV}{\mathrm{LV}}
\newcommand{\ii}{{\mathrm{i}}}
\newcommand{\A}{\mathrm{A}}
\newcommand{\Energy}{\mathcal{E}}

\declaretheorem{theorem}
\declaretheorem[sibling=theorem]{proposition}
\declaretheorem[sibling=theorem]{corollary}
\declaretheorem[sibling=theorem]{remark}
\declaretheorem[sibling=theorem]{lemma}
\declaretheorem[sibling=theorem]{definition}
\declaretheorem[sibling=theorem]{example}
\declaretheorem[sibling=theorem]{conjecture}
\declaretheorem[sibling=theorem]{heuristic}

\title[Exponent pairs and zero density estimates]{New exponent pairs, zero density estimates, and zero additive energy estimates: a systematic approach}
\author[T. Tao]{Terence Tao}
\address{UCLA Department of Mathematics, Los Angeles, CA 90095-1555.}
\email{tao@math.ucla.edu}
\thanks{Supported by NSF grant DMS-2347850.}

\author[T. S. Trudgian]{Tim Trudgian}
\address{School of Science, The University of New South Wales, Canberra, Australia}
\email{timothy.trudgian@unsw.edu.au}
\thanks{Supported by ARC grant DP240100186.}

\author[A. Yang]{Andrew Yang}
\address{School of Science, The University of New South Wales, Canberra, Australia}
\email{andrew.yang1@unsw.edu.au}

\begin{document}
\date\today

\begin{abstract}
    We obtain several new bounds on exponents of interest in analytic number theory, including four new exponent pairs, new zero density estimates for the Riemann zeta-function, and new estimates for the additive energy of zeroes of the Riemann zeta-function.  These results were obtained by creating the \emph{Analytic Number Theory Exponent Database} (ANTEDB) to collect results and relationships between these exponents, and then systematically optimising these relationships to obtain the new bounds.  We welcome further contributions to the database, which aims to allow easy conversion of new bounds on these exponents into optimised bounds on other related exponents of interest.
    \end{abstract}
\maketitle
\section{Introduction}
In this paper we introduce the \emph{Analytic Number Theory Exponent Database} (ANTEDB)\footnote{\url{https://teorth.github.io/expdb/}.}, which aims to be a living database that systematically collects, abstracts, develops code for, and optimises results and relations for, a number of exponents that arise in analytic number theory.   By an \emph{exponent}, we mean one or more real numbers, possibly depending on other exponent parameters, that occur as an exponent in an analytic number theory estimate, for instance as the exponent in some scale parameter $T$ that bounds some other quantity of interest.

For the initial launch of the ANTEDB, we focused on exponents relating to the Riemann zeta-function $\zeta(s)$. More specifically, in this paper we shall restrict attention to the exponents listed in the following table.

\begin{center}
    \begin{tabular}{lll}
        Exponent & Informal description & Key conjecture \\
        \hline
        $(k,\ell)$ & Exponent pairs & Exponent pair conj. \\
        $\beta(\alpha)$ & Dual function to exponent pairs & Exponent pair conj. \\
        $\mu(\sigma)$ & Growth exponent for $\zeta$ & Lindel\"{o}f Hypothesis \\
        $\LV(\sigma,\tau)$ & Large value exponents & Montgomery conj. \\
        $\LV_\zeta(\sigma,\tau)$ & Large value exponents for $\zeta$ & Lindel\"{o}f Hypothesis \\
        $\A(\sigma)$ & Zero density exponent for $\zeta$ & Density Hypothesis \\
        $\LV^*(\sigma,\tau)$ & Large value additive energy exp. & --- \\
        $\LV^*_\zeta(\sigma,\tau)$ & Large value additive energy exp. for $\zeta$ & Lindel\"{o}f Hypothesis \\
        $\A^*(\sigma)$ & Zero density additive energy exp. for $\zeta$ & --- \\
        $\Energy$ & Energy / double zeta sum tuples & --- \\
        $\Energy_\zeta$ & Energy / double zeta sum tuples for $\zeta$ & --- \\
        \hline
    \end{tabular}
\end{center}

We will review the definitions of these exponents in Section \ref{notation-sec}.  The exponents $(k,\ell)$, $\mu(\alpha)$, $\A(\sigma)$, (and to a lesser extent, $\beta(\alpha)$) are well known\footnote{There are some minor variations in notation in the literature regarding whether to permit ``epsilon losses'' or to define the class of phase functions to which exponent pairs apply.  See \Cref{epsilon-loss} and Example \ref{phase-ex} below.} and extensively studied as explicit quantities of interest in many places in the literature (e.g., in \cite{ivic}).  However, many of the other exponents in the above table are only studied \emph{implicitly}, with specific values or bounds on these exponents often appearing in auxiliary estimates  used to control one of the more explicitly studied exponents listed above.  One of the objectives of the ANTEDB is to make these exponents more explicit, and to identify and abstract the relationships between them that often only appear in the literature in special cases (for instance, with certain parameters set to specific numerical values).

The ANTEDB also contains some preliminary data on several other related exponents, such as those relating to gaps between primes. However these are less developed at present, and we will not report on them in depth here.

There is a vast literature on bounding these various exponents, and on relating them to each other.  One  resource for this literature is the text of Ivi\'{c} \cite{ivic}, and many of the most recent bounds may be found in a paper \cite{trudgian-yang} by the second and third authors.  However, due to a steady stream of new results in the field (for instance, the spectacular zero density estimates of Guth and Maynard \cite{guth-maynard}), any given list of bounds in the literature will eventually become out of date.  Furthermore, many of the key relations between exponents are expressed in the literature in a way that is tailored to the best numerical exponents available at the time of publication, rather than in a more general and abstract fashion that could be easily applied to subsequent work that can utilise improved exponent bounds.

To address these issues, we have attempted not only to collect all known bounds on the above exponents in the ANTEDB, but also to ``unpack'' the proofs of these exponents by abstracting out general relations between exponents that are implicit in the literature, but not stated in the most general form suitable for future applications.  A simple and well known example is the abstraction of the functional equation of the Riemann zeta-function as a symmetry relation
\begin{equation}\label{mu-func}
    \mu(1-\sigma) = \mu(\sigma) + \sigma - 1/2
\end{equation}
for the growth rates
\begin{equation}\label{mu-def}
    \mu(\sigma) \coloneqq \limsup_{t \to \infty} \frac{\log |\zeta(\sigma+it)|}{\log t}
\end{equation}
of the Riemann zeta-function (see, e.g., \cite[(1.23), (1.25)]{ivic}).  A more implicit, but also well known, example is the technique of \emph{Huxley subdivision}, which is a commonly used method in the theory of zero density estimates for extending large value estimates on short intervals to large value estimates on long intervals.  By introducing the large value exponent $\LV(\sigma,\tau)$, which we define in \Cref{lv-def}, this subdivision method can be abstracted as the upper Lipschitz bound
\begin{equation*}\label{subdivision}
 \LV(\sigma,\tau') \leq \LV(\sigma,\tau) + \tau' - \tau
\end{equation*}
whenever $1/2 \leq \sigma \leq 1$ and $\tau \leq \tau'$; see \Cref{hux-sub}.  By identifying these abstract relations, one can automate some of the more ``routine'' processes in the literature, such as the ``standard arguments'' appearing in \cite{guth-maynard} to pass from large value theorems to zero density estimates (and thence to bounds on gaps between primes).

Such abstract relations combine to produce a complicated network of implications, which often hinders manual optimisation an exponent bound given a set of known results. For instance, an exponent pair $(k, \ell)$ (defined in Definition \ref{exp-pair-def}) may be implied by other exponent pairs via a number of processes, as well as bounds on various exponential sums that themselves depend on exponent pairs. However, this paradigm appears to be well-suited for machine-assisted optimisation. In ANTEDB we implement a python module that implements the set of known results and their relations, as well as routines to find new results implied by a set of hypotheses. A number of results recorded in this article were found in this way. 

Thus, each theorem and conjecture in ANTEDB is represented both in human-readable form as well as python code. As the literature of exponent bounds grows, one may increasingly rely on machine-assistance to  maintain an up-to-date record of the best known bounds of various exponents, which may either be explicitly derived or implied by novel upstream results. The aim of the ANTEDB python module is to abstract away the tedious task of computing the consequences of a novel intermediate result, thereby freeing up resources for researchers. For instance, if a new large value estimate is discovered, one may use the ANTEDB python module to find efficiently the best-possible new zero density estimates that are implied, and compare these to zero density estimates already in the database.

\begin{remark}[General note on notation]\label{epsilon-loss}  Throughout this paper, we adopt notational conventions for exponents that allow for ``epsilon losses'' in the estimates.  This is done in order to simplify the presentation of the results (which would otherwise require various additional $\eps$ terms), and for compatibility with the code in the ANTEDB, in which exponents are frequently stored as rational numbers to facilitate exact arithmetic.  Several of our the results  may permit one or more of these epsilons to be replaced with more precise losses such as a logarithmic loss, or even with a completely ``log-free'' estimate, but we have not pursued these refinements here or in the ANTEDB.  It is possible however that a future expansion of the ANTEDB could also address these more refined questions.
\end{remark}

\subsection{New results}

As a demonstration of the effectiveness of this systematic approach, we were able to obtain new exponent bounds purely from combining and optimising existing methods, without the introduction of new analytic number theory inputs.  The discovery of these bounds was very much computer assisted; however, in many cases we were able to extract human-readable proofs from these computations, which we will present in this paper.  We expect that future improvements on these exponents coming from new techniques in analytic number theory can be similarly combined with the machinery in the ANTEDB to allow these improvements to propagate in an optimised and largely automated fashion.

Our new results can be summarised as follows:
\begin{itemize}
\item Four new exponent pairs (\Cref{new-exp-pair});
\item Several new zero density estimates (\Cref{zero_density_estimates_table}); and
\item New estimates for the additive energy of zeroes (see \Cref{Add-est}).
\end{itemize}
\subsection{Future directions}

We welcome contributions from the analytic number theory community to maintain and expand the ANTEDB.  Some natural directions of expansion could include recording analogous exponents for various $L$-functions, or to record estimates that only lose logarithmic factors rather than epsilon exponent losses, or even ``log-free'' or explicit estimates.  In the latter aspect, it may eventually be possible to unify this database with the Th\'eorie Multiplicative Explicite des nombres / Explicit Multiplicative Number Theory website at \url{https://tmeemt.github.io/Chest/}.

Currently, the ANTEDB python module performs computations using routines that are not formally certified to be error-free.  A natural future direction of the project would be to incorporate formal verification in languages such as \emph{Lean}.  Completely formalizing the estimates in the ANTEDB would be a significant challenge\footnote{For instance, the ongoing ``Prime Number Theorem and more'' project at \url{https://github.com/AlexKontorovich/PrimeNumberTheoremAnd}, launched in January 2024, took approximately four months to formalise the prime number theorem with qualitative error term, with the classical error term bound still a work in progress at this time of writing.}; however, several conditional calculations deriving one exponent from other exponents in the literature could conceivably be formalised within the ANTEDB.

\section{Notation}

\subsection{Asymptotic notation}\label{notation-sec}

It is convenient to use a ``cheap non-standard analysis'' framework for asymptotic notation, in the spirit of \cite{tao-cheap}, as this will reduce the amount of ``epsilon management'' one has to do in the arguments.  This framework is inspired by non-standard analysis, but we will avoid explicitly using non-standard constructions such as ultraproducts in the discussion below, relying instead on the more familiar notion of sequential limits.  Any result stated here using this non-standard framework can be converted without difficulty to a result in standard ``epsilon-delta'' form; see for instance \Cref{beta-asymp} for a typical example.

In this framework, we assume there is some ambient index parameter $\ii$, which ranges over some ambient sequence of natural numbers going to infinity.  All mathematical objects $X$ (numbers, sequences, sets, functions, etc.), will either be \emph{fixed} --- i.e., independent of $\ii$ --- or \emph{variable}\footnote{These correspond to the notions of \emph{standard} and \emph{non-standard} objects in non-standard analysis.} - i.e., dependent on $\ii$.    Of course, fixed objects can be considered as special cases of variable objects, in which the dependency is constant.  By default, objects should be understood to be variable if not explicitly declared to be fixed. For emphasis, we shall sometimes write $X = X_{\ii}$ to indicate explicitly that an object $X$ is variable; however, to reduce clutter, we shall generally omit mention of the parameter $\ii$ in most of our arguments. We will often reserve the right to refine the ambient sequence to a subsequence as needed, usually in order to apply a compactness theorem such as the Bolzano--Weierstrass theorem; we refer to this process as ``passing to a subsequence if necessary''.  When we say that a statement involving variable objects is true, we mean that it is true for all $\ii$ in the ambient sequence.
For instance, a variable set $E$ of real numbers is a set $E = E_{\ii}$ indexed by the ambient parameter $\ii$, and by an element of such a set, we mean a variable real number $x = x_{\ii}$ such that $x_{\ii} \in E_{\ii}$ for all $\ii$ in the ambient sequence.

We isolate some special types of variable numerical quantities $X = X_{\ii}$ (which could be a natural number, real number, or complex number):
\begin{itemize}
\item $X$ is \emph{bounded} if there exists a fixed $C$ such that $|X| \leq C$. In this case we also write $X = O(1)$.
\item $X$ is \emph{unbounded} if $|X_{\ii}| \to \infty$ as $\ii \to \infty$; equivalently, for every fixed $C$, one has $|X| \geq C$ for $\ii$ sufficiently large.
\item $X$ is \emph{infinitesimal} if $|X_{\ii}| \to 0$ as $\ii \to \infty$; equivalently, for every fixed $\eps>0$, one has $|X| \leq \eps$ for $\ii$ sufficiently large. In this case we also write $X = o(1)$.
\end{itemize}

Note that any quantity $X$ will be either bounded or unbounded, after passing to a subsequence if necessary; similarly, by the Bolzano--Weierstrass theorem, any bounded (variable) quantity $X$ will be of the form $X_0+o(1)$ for some fixed $X_0$, after passing to a subsequence if necessary.  Thus, for instance, if $T, N > 1$ are (variable) quantities with $N = T^{O(1)}$ (or equivalently, $T^{-C} \leq N \leq T^C$ for some fixed $C$), then, after passing to a subsequence if necessary, we may write $N = T^{\alpha+o(1)}$ for some fixed real number $\alpha$.  Note that any further passage to subsequences does not affect these concepts; quantities that are bounded, unbounded, or infinitesimal remain so under any additional restriction to subsequences.

We observe the \emph{underspill principle}: if $X,Y$ are (variable) real numbers, then the relation
$$ X \leq Y + o(1)$$
is equivalent to the relation
$$ X \leq Y + \eps + o(1)$$
holding for all fixed $\eps>0$.

We can develop other standard asymptotic notation in the natural fashion: given two (variable) quantities $X,Y$, we write $X = O(Y)$, $X \ll Y$, or $Y \gg X$ if $|X| \leq CY$ for some fixed $C$, and $X = o(Y)$ if $|X| \leq cY$ for some infinitesimal $c$.  We also write $X \asymp Y$ for $X \ll Y \ll X$.

A convenient property of this asymptotic formalism, analogous to the property of \emph{$\omega$-saturation} in non-standard analysis, is that certain asymptotic bounds are automatically uniform in variable parameters.

\begin{proposition}[Automatic uniformity]\label{auto}  Let $E = E_{\ii}$ be a non-empty variable set, and let $f = f_{\ii}: E \to \C$ be a variable function.  \begin{itemize}
    \item[(i)] Suppose that $f(x)=O(1)$ for all (variable) $x \in E$.  Then after passing to a subsequence if necessary, the bound is uniform, that is to say, there exists a fixed $C$ such that $|f(x)| \leq C$ for all $x \in E$.
    \item[(ii)] Suppose that $f(x)=o(1)$ for all (variable) $x \in E$.  Then after passing to a subsequence if necessary, the bound is uniform, that is to say, there exists an infinitesimal $c$ such that $|f(x)| \leq c$ for all $x \in E$.
\end{itemize}
\end{proposition}

\begin{proof} We begin with (i).  Suppose that there is no uniform bound.  Then for any fixed natural number $n$, one can find arbitrarily large $\ii_n$ in the sequence and $x_{\ii_n} \in E_{\ii_n}$ such that $|f_{\ii_n}(x_{\ii_n})| \geq n$.  Clearly one can arrange matters so that the sequence $\ii_n$ is increasing.  If one then restricts to this sequence and sets $x$ to be the variable element $x_{\ii_n}$ of $E$, then $f(x)$ is unbounded, a contradiction.

Now we prove (ii).  We can assume for each fixed $n$ that there exists $\ii_n$ in the ambient sequence such that $|f_{\ii}(x_{\ii})| \leq 1/n$ for all $\ii \geq \ii_n$ and $x_{\ii} \in E_{\ii}$, since if this were not the case one can construct an $x = x_{\ii} \in E$ such that $|f_{\ii}(x_{\ii})| \geq 1/n$ for $\ii$ sufficiently large, contradicting the hypothesis.  Again, we may take the $\ii_n$ to be increasing.  Restricting to this sequence, and writing $c_{\ii_n} \coloneqq 1/n$, we see that $c=o(1)$ and $|f(x)| \leq c$ for all $x \in E$, as required.
\end{proof}

\begin{remark} It is important in Proposition \ref{auto} that the hypotheses in (i), (ii) are assumed for all \emph{variable} $x \in E$, rather than merely the \emph{fixed} $x \in E$.  For instance, let $E = \R$ and consider the variable function $f_{\ii}(x) \coloneqq x/\ii$.  Then $f(x) = o(1)$ for any fixed $x \in E$, but the decay rate is not uniform, and we do not have $f(x) = o(1)$ for all variable $x \in E$ (e.g., $x_\ii \coloneqq \ii$ is a counterexample).
\end{remark}

\begin{remark} There are two caveats with this asymptotic formalism.  First, the law of the excluded middle is only valid after passing to subsequences.  For instance, it is possible for a non-standard natural number to be neither even nor odd, since it could be even for some $\ii$ and odd for others.  However, one can pass to a subsequence in which it becomes either even or odd.  Second, one cannot combine the ``external'' concepts of asymptotic notation with the ``internal'' framework of (variable) set theory.  For instance, one cannot view the collection of all bounded (variable) real numbers as a variable set, since the notion of boundedness is not ``pointwise'' to each index $\ii$, but instead describes the ``global'' behavior of this index set.  Thus, for instance, set builder notation such as $\{ x: x = O(1) \}$ should be avoided.
\end{remark}

\subsection{Phase functions}

As a first use of this asymptotic notation, we define our notion of phase functions used in the concept of an exponential pair.  For simplicity we restrict attention to phase functions that are (asymptotically) ``monomial functions'' in the sense of \cite[Chapter 3]{huxley_area_1996}, although many of the results discussed here can likely be generalised to larger classes of phase functions.

\begin{definition}[Phase function]\label{phase-def}
    A \emph{phase function} is a (variable) smooth function $F \colon [1,2] \to \R$.  A phase function $F$ will be called a \emph{model phase function} if there exists a fixed exponent $\sigma > 0$ with the property that
    \begin{equation}\label{fpu}
    F^{(p+1)}(u) - \frac{d^p}{du^p} u^{-\sigma} = o(1)
    \end{equation}
    for all (variable) $u \in [1,2]$ and all fixed $p \geq 0$, where $F^{(p+1)}$ denotes the $(p+1)^{\mathrm{st}}$ derivative of $F$.
    \end{definition}

    \begin{example}\label{phase-ex} The function $u \mapsto \log u$ is a model phase function (with $\sigma=1$), and for any fixed $\sigma \neq 1$, $u \mapsto u^{1-\sigma}/(1-\sigma)$ is also a model phase function.  Informally, a model phase function is a function which asymptotically behaves like $u \mapsto \log u$ (for $\sigma = 1$) or $u \mapsto u^{1-\sigma}/(1-\sigma)$ (for $\sigma \neq 1$), up to constants.  This turns out to be a good class for exponential sum estimates, as it is stable under Weyl differencing and Legendre transforms, which show up in the van der Corput A-process and B-process respectively.  Results on exponential pairs can often extend to more general classes of phase functions, but this class is sufficient for applications to the Riemann zeta-function, and so we will not study any class beyond the model phase function class here or in the ANTEDB.  (However, such generalisations could be addressed in a potential future expansion of the ANTEDB.)
    \end{example}

    Note from Proposition \ref{auto} that the $o(1)$ decay rate in \eqref{fpu} can be made uniform, after passing to a subsequence if necessary.

\subsection{Additive energy}\label{add-sec}

Additive energy is a concept that appears implicitly in the work of Heath-Brown \cite{heath_brown_consecutive_II}, although the terminology of additive energy\footnote{Strictly speaking, only the scale zero energy $E_0(W)$ was introduced in \cite{tao-vu}, but the concepts are related.  For instance, if $W'$ is the (multi-)set formed from $W$ by rounding to the nearest integer, then $E_1(W)$ can be shown to be comparable to $E_0(W')$, as can be seen from \eqref{energy-asymp} below.} first appears in \cite{tao-vu}.  As observed in that paper, controlling the additive energy of zeroes of the Riemann zeta-function has implications for the distributions of primes in short intervals.  Additive energy of large values of Dirichlet series also played a key role in the recent zero density estimates in \cite{guth-maynard}.

\begin{definition}[Additive energy]\label{energy-def}  Let $W$ be a finite multi-set of real numbers.  The \emph{additive energy} $E_1(W)$ of such a set is defined to be the number of quadruples $(t_1,t_2,t_3,t_4) \in W$ such that
    $$    |t_1 + t_2 - t_3 - t_4| \leq 1.$$
\end{definition}

One could more generally define $E_r(W)$ for any $r>0$ to be the number of quadruples $(t_1,t_2,t_3,t_4) \in W$ such that
$$    |t_1 + t_2 - t_3 - t_4| \leq r,$$
but from \cite[Lemma 3.1]{cladek-tao} we have the comparability relation
\begin{equation}\label{energy-asymp}
    E_1(W) \asymp E_r(W) \asymp \bigg\| \sum_{t \in W} 1_{[t-r,t+r]} \bigg\|_{U^2}^4
\end{equation}
for any fixed $r>0$, where the \emph{Gowers uniformity norm} $\|f\|_{U^2}$ was defined as
$$\|f\|_{U^2}^4 \coloneqq \int_\R \int_\R \int_\R f(x) \bar{f}(x+h) \bar{f}(x+k) f(x+h+k)\ dx\ dh\ dk.$$
As a consequence, one has a quasi-triangle inequality
\begin{equation}\label{energy-triangle}
    E_1\left(\bigcup_{j=1}^J W_j\right)^{1/4} \ll \sum_{j=1}^J E_1(W_j)^{1/4}
\end{equation}
for any (multi-)sets $W_1,\dots,W_J \subset \R$.

\subsection{Double zeta sums}

Double zeta sums are a type of exponential sum that arise frequently in the study of large value theorems and zero density estimates.
Given a (multi-)set $W$ and a scale $N>1$, let $S(N,W)$ denote the \emph{double zeta sum}
\begin{equation*}\label{snw-def}
     S(N,W) \coloneqq \sum_{t,t' \in W} \bigg|\sum_{n \in [N,2N]} n^{-i(t-t')} \bigg|^2.
\end{equation*}
We caution that this normalisation differs from the one in \cite{ivic}, where $n^{-1/2-i(t-t')}$ is used in place of $n^{-i(t-t')}$.  This sum may also be rearranged as
\begin{equation}\label{snw}
 S(N,W) = \sum_{n,m \in [N,2N]} |R_W(n/m)|^2
\end{equation}
where $R_W$ is the exponential sum
$$ R_W(x) \coloneqq \sum_{t \in W} x^{it}.$$
From this formulation of the double zeta sum, we have the triangle inequality
\begin{equation*}\label{snw-triangle}
    S\left(N,\bigcup_{j=1}^J W_j\right)^{1/2} \leq \sum_{j=1}^J S(N,W_j)^{1/2}
\end{equation*}
for any (multi-)sets $W_1,\dots,W_J \subset \R$.

\begin{remark} In \cite[Section 2]{guth-maynard}, triple zeta sums of the form
$$ \sum_{t_1,t_2,t_3 \in W} \sum_{n_1,n_2,n_3 \in [N,2N]} n_1^{i(t_1-t_2)} n_2^{i(t_2-t_3)} n_3^{i(t_3-t_1)}$$
play a key role in the arguments (actually for technical reasons some diagonal portions of this sum, in which two of the $t_1,t_2,t_3$ are close to each other, are removed).  In principle, estimates on such sums could also be systematically collected by the ANTEDB, but we have not attempted to do this.
\end{remark}

\section{Exponential sum exponents, and exponent pairs}

We now use the concept of a phase function to define the exponent sum growth exponent function $\beta(\alpha)$.

    \begin{definition}[Exponent sum growth exponent]\label{beta-def}  For any fixed $\alpha \geq 0$, let $\beta(\alpha) \in \R$ denote the least possible (fixed) exponent for which the following claim holds: whenever $N, T \geq 1$ are (variable) quantities with $T$ unbounded and $N = T^{\alpha+o(1)}$, $F$ is a model phase function, and $I \subset [N, 2N]$ is an interval, then
        $$ \sum_{n \in I} e(T F(n/N)) \ll T^{\beta(\alpha)+o(1)}.$$
    \end{definition}

    It is easy to see that the set of possible candidates for $\beta(\alpha)$ is closed (thanks to underspill), non-empty, and bounded from below, so $\beta$ is well-defined as a (fixed) function from $[0,+\infty)$ to $\R$.  Specializing to the logarithmic phase $F(u) = \log u$, and performing a complex conjugation, we see in particular that
\begin{equation*}\label{beta-alpha}
    \sum_{n \in I} n^{-iT} \ll T^{\beta(\alpha)+o(1)}
\end{equation*}
whenever $T$ is unbounded, $N = T^{\alpha+o(1)}$, and $I$ is an interval in $[N,2N]$.  Thus it is clear that knowledge of $\beta$ is  relevant to understanding the Riemann zeta-function.

The quantity $\beta(\alpha)$ can also be formulated without asymptotic notation, but at the cost of introducing some ``epsilon and delta'' parameters:

\begin{lemma}[Non-asymptotic definition of $\beta$]\label{beta-asymp}  Let $\alpha \geq 0$ and $\overline{\beta} \in \R$ be fixed.  Then the following are equivalent:
\begin{itemize}
\item[(i)] $\beta(\alpha) \leq \overline{\beta}$.
\item[(ii)] For every (fixed) $\eps>0$ and $\sigma > 0$ there exists (fixed) $\delta>0$, $P \geq 1$, $C \geq 1$ with the following property: if $T \geq C$, $T^{\alpha-\delta} \leq N \leq T^{\alpha+\delta}$ are (fixed) real numbers, $I \subset [N,2N]$ is a (fixed) interval, and $F$ is a (fixed) phase function such that
\begin{equation}\label{fpu-bound}
\bigg|F^{(p+1)}(u) - \frac{d^p}{du^p} u^{-\sigma}\bigg| \leq \delta
\end{equation}
for all (fixed) $0 \leq p \leq P$ and $u \in [1,2]$, then
$$ \bigg|\sum_{n \in I} e(T F(n/N))\bigg| \leq C T^{\overline{\beta}+\eps}.$$
\end{itemize}
\end{lemma}

\begin{proof}  It is easy to see that (ii) implies (i) by expanding out all the definitions (and using Proposition \ref{auto} to resolve any uniformity issues).  Conversely, suppose that (ii) fails.  Carefully negating all the quantifiers, we conclude that there exists a fixed $\eps, \sigma > 0$ such that for any (fixed) natural number $\ii$, one can find real numbers $T = T_{\ii} \geq \ii$, $T^{\alpha-1/\ii} \leq N = N_{\ii} \leq T^{\alpha+1/\ii}$, an interval $I = I_{\ii} \subset [N_\ii, 2N_\ii]$, and a phase function $F = F_{\ii}$ such that
$$ \bigg|F_\ii^{(p+1)}(u) - \frac{d^p}{du^p} u^{-\sigma}\bigg| \leq 1/\ii$$
for all (fixed) $0 \leq p \leq \ii$ and $u \in [1,2]$, but that
$$ \bigg|\sum_{n \in I} e(T F(n/N))\bigg| \geq \ii T^{\overline{\beta}+\eps}.$$
But then $F = F_\ii$ is a model phase function which gives a counterexample to the claim $\beta(\alpha) \leq \overline{\beta}$.
\end{proof}

It will be slightly more convenient technically to work with the asymptotic formulation in this paper. Furthermore, we will primarily be interested in $\beta$ in the range $0 \leq \alpha \leq 1$; for $\alpha > 1$ one can use the Euler--Maclaurin formula (see e.g. \cite[(2.1.2)]{titchmarsh_theory_1986}) to obtain that $\beta(\alpha)=\alpha-1$ in this regime.

From the $L^2$ mean value theorem (see, e.g. \cite[Theorem 9.1]{ik}) we have
$$ \int_{T}^{2T} \bigg|\sum_{n \in [N,2N]} e(t \log(n/N)) \bigg|^2\ dt \asymp T N$$
from which we can conclude the lower bound
$$\beta(\alpha) \geq \alpha/2$$
for $0 \leq \alpha \leq 1$.  Conjecturally, this lower bound is sharp; see \Cref{exp-pair-conj} below.  From the van der Corput inequalities (see \cite[Theorem 8.20]{ik}) we can verify this conjecture at the endpoints, thus
\begin{equation}\label{beta-end}
    \beta(0) = 0; \quad \beta(1) = 1/2
\end{equation}
while from the van der Corput $B$-process (see, e.g., \cite[p 370]{huxley_area_1996}) we have the symmetry
\begin{equation}\label{beta-reflect}
\beta(1-\alpha) = \frac{1}{2} - \alpha + \beta(\alpha)
\end{equation}
for $0 \leq \alpha \leq 1$.

The function $\beta$ is traditionally studied through the concept of an \emph{exponent pair}.

\begin{definition}[Exponent pair]\label{exp-pair-def}  An exponent pair is a (fixed) element $(k,\ell)$ of the triangle
    \begin{equation}\label{exp-pair-triangle}
        \{ (k,\ell) \in \R^2: 0 \leq k \leq 1/2 \leq \ell \leq 1, k+\ell \leq 1 \}
    \end{equation}
    with the following property: for all model phase functions $F$, all $T \geq N \leq 1$, and all intervals $I \subset [N,2N]$, one has
    \begin{equation*}\label{ntf}
     \sum_{n \in I} e(T F(n/N)) \ll (T/N)^{k+o(1)} N^{\ell+o(1)}
    \end{equation*}
    whenever $T \geq N \geq 1$, $I$ is an interval in $[N,2N]$, and $F \in {\mathcal U}$.
\end{definition}

Again, we can phrase this concept non-asymptotically:

\begin{lemma}[Non-asymptotic definition of exponent pair]  Let $(k,\ell)$ be a fixed element of \eqref{exp-pair-triangle}.  Then the following are equivalent:
    \begin{itemize}
    \item[(i)] $(k,\ell)$ is an exponent pair.
    \item[(ii)] For every (fixed) $\eps>0$ there exist (fixed) $C, P > 0$ such that, whenever $T \geq N \geq 1$, $I \subset [N,2N]$, and $F$ is a phase function obeying \eqref{fpu-bound} for for all (fixed) $0 \leq p \leq P$ and $u \in [1,2]$, then
    \begin{equation*}\label{tfn}
        \bigg|\sum_{n \in I} e(T F(n/N))\bigg| \leq C (T/N)^{k+\eps} N^{\ell+\eps}.
    \end{equation*}
   \end{itemize}
\end{lemma}

The proof of this lemma is similar to that of Lemma \ref{beta-asymp} and is omitted.

We record some known processes that preserve exponent pairs:

\begin{lemma}[Exponent pair processes]\label{exp-process}  The space of exponent pairs is preserved under the following processes:
    \begin{itemize}
    \item[(A)] The $A$-process $A: (k,\ell) \mapsto (\frac{k}{2k+2}, \frac{\ell}{2k+2} + \frac{1}{2})$;
    \item[(B)] The $B$-process $B: (k,\ell) \mapsto (\ell-\frac{1}{2}, k+\frac{1}{2})$;
    \item[(C)] The $C$-process $C: (k,\ell) \mapsto (\frac{k}{12(1+4k)}, \frac{11(1+4k)+\ell}{12(1+4k)})$.
    \end{itemize}
\end{lemma}

\begin{proof} For (A), (B), see for instance \cite[Lemmas 2.8, 2.9]{ivic}.  For (C), see \cite[Theorem 5]{sargos_analog_2003}. 
\end{proof}

There exists a further process which almost preserves the space of exponent pairs, which may be stated in terms of bounds on $\beta$. 

\begin{lemma}[Sargos $D$-process]\label{D-process} If $(k, \ell)$ is an exponent pair, then $$\beta(\alpha) \le \max\left(k_1 + \alpha(\ell_1 - k_1), \frac{1}{12} + \frac{2}{3}\alpha\right)$$ for $0\le \alpha \le 1$, where 
\[
(k_1, \ell_1) = D(k, \ell) := \Big(\frac{5k+\ell+2}{8(5k+3\ell+2)}, \frac{29k+21\ell+10}{8(5k+3\ell+2)}\Big).
\]
\end{lemma}
\begin{proof}
See \cite[Theorem 7.1]{sargos_points_1995}.
\end{proof}

In practice, this lemma  implies that if $(k, \ell)$ is an exponent pair, then so is $D(k, \ell)$. 

The relation between exponent pairs and the function $\beta$ can be summarised as follows.

\begin{lemma}[Duality between exponent pairs and $\beta$]\label{beta-duality}  Let $(k,\ell)$ be in the triangle \eqref{exp-pair-triangle}.  Then the following are equivalent:
    \begin{itemize}
    \item[(i)] $(k,\ell)$ is an exponent pair.
    \item[(ii)] $\beta(\alpha) \leq k + (\ell-k)\alpha$ for all $0 \leq \alpha \leq 1$.
    \end{itemize}
\end{lemma}

\begin{proof}  If (i) holds, then for any $0 < \alpha < 1$, any unbounded $T \geq 1$, any $N = T^{\alpha+o(1)}$, interval $I \subset [N,2N]$, and model phase function $F$, we have from (i) that
$$ \sum_{n \in I} e(T F(n/N)) \ll (T/N)^{k+o(1)} N^{\ell+o(1)} = T^{k + (\ell-k)\alpha + o(1)}.$$
From Definition \ref{beta-def} we conclude that $\beta(\alpha) \leq k + (\ell-k) \alpha$.  Also since $(k,\ell)$ lies in \eqref{exp-pair-triangle}, we see from \eqref{beta-end} that we also have $\beta(\alpha) \leq k + (\ell-k) \alpha$ for $\alpha=0,1$.

Now suppose that (ii) holds.  Let $F, T, N, I$ be as in Definition \ref{exp-pair-def}.  By underspill it suffices to show that
$$ \sum_{n \in I} e(T F(n/N)) \ll (T/N)^{k+\eps+o(1)} N^{\ell+\eps+o(1)}$$
for any fixed $\eps>0$.  We may assume that $T \leq N^{1/\eps+1}$, since the claim follows from the trivial bound $\sum_{n \in I} e(T F(n/N)) \ll N$ otherwise.  We may also assume that $N$ is unbounded, since the claim is clear for $N$ bounded; this forces $T$ to be unbounded as well.

By passing to a subsequence we may assume that $N = T^{\alpha+o(1)}$ for some fixed $0 \leq \alpha \leq 1$.  By Definition \ref{beta-def} we then have
$$ \sum_{n \in I} e(T F(n/N)) \ll T^{\beta(\alpha)+o(1)}$$
and hence by (ii)
$$ \sum_{n \in I} e(T F(n/N)) \ll (T/N)^{k+o(1)} N^{\ell+o(1)}$$
giving the claim.
\end{proof}

\begin{remark}  When applying \Cref{beta-duality} in the case $\ell-k \geq 1/2$, it suffices in (ii) to restrict to the range $0 \leq \alpha \leq 1/2$, as the remaining range $1/2 \leq \alpha \leq 1$ is then covered by the symmetry \eqref{beta-reflect}.  Similarly, when applying \Cref{beta-duality} in the case $\ell-k \leq 1/2$, it suffices to restrict to the range $1/2 \leq \alpha \leq 1$.
\end{remark}

The main conjecture on exponent pairs and $\beta$ is then

\begin{conjecture}[Exponent pair conjecture]\label{exp-pair-conj}\
    \begin{itemize}
    \item[(i)] Every point in the triangle \eqref{exp-pair-triangle} is an exponent pair.  In particular, $(0,1/2)$ is an exponent pair.
    \item[(ii)]  We have
\[
\beta(\alpha) = \begin{cases}
\alpha/2,& 0 \leq \alpha \leq 1\\
\alpha - 1,&\alpha > 1
\end{cases}.
\]
    \end{itemize}
\end{conjecture}

In view of \Cref{beta-duality}, the two components of this conjecture are equivalent.

A large number of bounds on exponential sums can be interpreted as piecewise-linear upper bounds on $\beta(\alpha)$ in various ranges of $\alpha$; see e.g. \cite[Table 3]{trudgian-yang} (which in turn builds upon \cite[Tables 17.1, 19.2]{huxley_area_1996}). For many of these results, a larger class of phase functions can be treated but we ignore this direction of generalisation for simplicity.  There are additional bounds in the literature that are not covered by these tables.  To give just one example, we present

\begin{lemma}[Heath-Brown exponent bound]\label{heath-brown-2017}
$$ \beta(\alpha) \leq \alpha + \max\left( \frac{1-k\alpha}{k(k-1)}, -\frac{\alpha}{k(k-1)}, -\frac{2\alpha}{k(k-1)} - \frac{2(1-k\alpha)}{k^2(k-1)}\right)$$
for any $\alpha>0$ and any natural number $k \geq 3$.
\end{lemma}

\begin{proof}
Let $T \geq 1$ be unbounded, let $N = T^{\alpha+o(1)}$, let $F$ be a model phase function, and $I \subset [N, 2N]$ be an interval.  Writing $f(x) \coloneqq TF(x/N)$, we see from \eqref{fpu} that
$$ 0 < \lambda_k < f^{(k)}(x) \leq A \lambda_k$$
for all $x \in [N,2N]$ and some $\lambda_k \asymp T N^{-k} = T^{1-k\alpha+o(1)}$ and $A \asymp 1$.  Applying \cite[Theorem 1]{heathbrown_new_2017}, we conclude that
$$ \sum_{n \in I} e(T F(n/N)) \ll N^{1+\eps} ( \lambda_k^{1/k(k-1)} +N^{-1/k(k-1)} + N^{-2/k(k-1)} \lambda_k^{-2/k^2(k-1)})$$
for any fixed $\eps>0$, and hence by underspill
$$ \sum_{n \in I} e(T F(n/N)) \ll N^{1+o(1)} ( \lambda_k^{1/k(k-1)} +N^{-1/k(k-1)} + N^{-2/k(k-1)} \lambda_k^{-2/k^2(k-1)}).$$
Expressing the right-hand side as a power of $T$, we obtain
$$ \sum_{n \in I} e(T F(n/N)) \ll T^{\beta+o(1)}$$
where
$$ \beta \coloneqq \alpha + \max\left( \frac{1-k\alpha}{k(k-1)}, -\frac{\alpha}{k(k-1)}, -\frac{2\alpha}{k(k-1)} - \frac{2(1-k\alpha)}{k^2(k-1)}\right).$$
The claim now follows from \Cref{beta-def}.
\end{proof}

\begin{remark}
From the proof of \Cref{heath-brown-2017} we see that the result in \cite[Theorem 1]{heathbrown_new_2017} is somewhat stronger than the bound on $\beta$ presented here, because that result handles more general phase functions than model phase functions.  However, in order to combine results from different papers as easily as possible, we have restricted attention to model phase functions, as discussed in \Cref{phase-ex}.
\end{remark}

In \cite{bourgain_decoupling_2017}, the exponent pair $\left(\dfrac{13}{84}, \dfrac{55}{84}\right)$ was established, while in
\cite[Lemma 1.1]{trudgian-yang}, the exponent pairs $\left(\dfrac{715}{10238}, \dfrac{7955}{10238}\right)$, $\left(\dfrac{4742}{38463}, \dfrac{35731}{51284}\right)$ were established.
One can then combine these with the $A,D$ processes from \Cref{exp-process} and \Cref{D-process} to obtain further pairs, which give bounds on $\beta$ thanks to \Cref{beta-duality}.  When combined with the bounds from \cite[Table 3]{trudgian-yang}, one obtains the upper bounds on $\beta(\alpha)$ in \Cref{beta-table} (with the bounds already listed in \cite[Table 3]{trudgian-yang} marked with an asterisk).

With this improved table of $\beta$ estimates, one obtains the following four new exponent pairs, which are plotted in \Cref{fig1} and \Cref{fig2}.

\begin{table}[ht]
    \resizebox{\textwidth}{!}{
    \def\arraystretch{1.8}
    \centering
    \begin{tabular}{|c|c|c|}
    \hline
    $\beta(\alpha)$ bound & $\alpha$ range & Reference\\
    \hline
    $\dfrac{1}{20} + \dfrac{3}{4}\alpha$ & $0\leq \alpha < \dfrac{1}{4}$ & Lemma \ref{heath-brown-2017} with $k = 5$\\
    \hline
    $\dfrac{19}{20}\alpha$ & $\dfrac{1}{4}\leq \alpha < \dfrac{890}{3277}$ & Lemma \ref{heath-brown-2017} with $k = 5$\\
    \hline
    $\dfrac{89}{2706} + \dfrac{2243}{2706}\alpha$ & $\dfrac{890}{3277}\leq \alpha < \dfrac{199}{716}$ & \cite[Table 17.1]{huxley_area_1996}*\\
    \hline
    $\dfrac{1}{66} + \dfrac{235}{264}\alpha$ & $\dfrac{120}{419}\leq \alpha < \dfrac{754}{2579}$ & \cite[Table 17.1]{huxley_area_1996}*\\
    \hline
    $\dfrac{9}{217} + \dfrac{1389}{1736}\alpha$ & $\dfrac{754}{2579}\leq \alpha < \dfrac{251324}{841245}$ & Exponent pair $\left(\dfrac{9}{217}, \dfrac{1461}{1736}\right) = AD\left(\dfrac{13}{84}, \dfrac{55}{84}\right)$*\\
    \hline
    $\dfrac{2371}{43205} + \dfrac{52209}{69128}\alpha$ & $\dfrac{251324}{841245}\leq \alpha < \dfrac{861996}{2811205}$ & \begin{tabular}{@{}c@{}}Exponent pair $\left(\dfrac{2371}{43205}, \dfrac{280013}{345640}\right)$ \\
    $= A\left(\dfrac{4742}{38463}, \dfrac{35731}{51284}\right)$\end{tabular}\\
    \hline
    $\dfrac{13}{146} + \dfrac{47}{73}\alpha$ & $\dfrac{861996}{2811205}\leq \alpha < \dfrac{87}{275}$ & \cite[Table 17.1]{huxley_area_1996}*\\
    \hline
    $\dfrac{11}{244} + \dfrac{191}{244}\alpha$ & $\dfrac{87}{275}\leq \alpha < \dfrac{423}{1295}$ & \cite[Table 17.1]{huxley_area_1996}*\\
    \hline
    $\dfrac{89}{1282} + \dfrac{454}{641}\alpha$ & $\dfrac{423}{1295}\leq \alpha < \dfrac{227}{601}$ & \cite[Table 17.1]{huxley_area_1996}*\\
    \hline
    $\dfrac{29}{280} + \dfrac{173}{280}\alpha$ & $\dfrac{227}{601}\leq \alpha < \dfrac{12}{31}$ & \cite[Table 17.1]{huxley_area_1996}*\\
    \hline
    $\dfrac{1}{32} + \dfrac{103}{128}\alpha$ & $\dfrac{12}{31}\leq \alpha < \dfrac{1508}{3825}$ & \cite[Table 17.1]{huxley_area_1996}*\\
    \hline
    $\dfrac{18}{199} + \dfrac{521}{796}\alpha$ & $\dfrac{1508}{3825}\leq \alpha < \dfrac{62831}{155153}$ & Exponent pair $\left(\dfrac{18}{199}, \dfrac{593}{796}\right) = D\left(\dfrac{13}{84}, \dfrac{55}{84}\right)$\\
    \hline
    $\dfrac{569}{2800} + \dfrac{1053}{2800}\alpha$ & $\dfrac{62831}{155153}\leq \alpha < \dfrac{143}{349}$ & \cite[Table 19.2]{huxley_area_1996}*\\
    \hline
    $\dfrac{491}{5530} + \dfrac{1812}{2765}\alpha $ & $\dfrac{143}{349}\le\alpha < \dfrac{263}{638}$ & \cite[Table 19.2]{huxley_area_1996}*\\
    \hline
    $\dfrac{113}{1345} + \dfrac{897}{1345}\alpha $ & $ \dfrac{263}{638} \le\alpha < \dfrac{1673}{4038}$ & \cite[Table 19.2]{huxley_area_1996}*\\
    \hline
    $\dfrac{2}{9} + \dfrac{1}{3} \alpha $ & $ \dfrac{1673}{4038} \le\alpha < \dfrac{5}{12}$ & \cite[(3.18)]{bourgain_decoupling_2017}* \\
    \hline
    $\dfrac{1}{12} + \dfrac{2}{3}\alpha$ & $\dfrac{5}{12}\leq \alpha < \dfrac{3}{7}$ & \cite[(3.18)]{bourgain_decoupling_2017}* \\
    \hline
    $\dfrac{13}{84} + \dfrac{1}{2}\alpha$ & $\dfrac{3}{7}\leq \alpha \leq \dfrac{1}{2}$ & Exponent pair $\left(\dfrac{13}{84}, \dfrac{55}{84}\right)$* \\
    \hline
    \end{tabular}
    }
    \caption{Bounds on $\beta(\alpha)$. With slight abuse of notation, $D(k, \ell)$ means the exponent pair derived from $\beta$ estimates obtained by applying the $D$ process to $(k, \ell)$. In each case one may verify that an exponent pair is in fact generated.}
    \label{beta-table}
\end{table}

\begin{theorem}[New exponent pairs]\label{new-exp-pair} The following are exponent pairs:
\[
\left(\frac{89}{1282}, \frac{997}{1282}\right),\quad \left(\frac{652397}{9713986}, \frac{7599781}{9713986}\right),
\]
\[
\left(\frac{10769}{351096}, \frac{609317}{702192}\right),\quad \left(\frac{89}{3478}, \frac{15327}{17390}\right).
\]
\end{theorem}

\begin{proof} The claim follows from \Cref{beta-duality} after some further computer calculation.
\end{proof}

The results in Table~\ref{beta-table} were found systematically using a programmable part of the ANTEDB implemented in Python, which we briefly introduce here. The library records theorems and conjectures from the literature as \texttt{Hypothesis} objects. Each \texttt{Hypothesis} either directly references a result from the literature, or contains a dependency set of other \texttt{Hypothesis} objects as well as a \texttt{proof} attribute that specifies how the dependencies are assembled into a proof. For instance, the exponent pair $(1/6, 2/3)$ can be represented in the database as a \texttt{Hypothesis} with two dependencies: one representing the $A$ process (with no dependencies of its own) and another representing the $(1/2, 1/2)$ exponent pair.

Equipped with a (large) set of \texttt{Hypothesis} objects representing known exponent pairs and bounds on $\beta$, as well as how they relate to one another, one may use a number of routines in the ANTEDB to combine those objects optimally. Here, we describe a polytope-based computation that may be easily generalised for optimising many other exponents recorded in the database. First, the starting \texttt{Hypothesis} set is repeatedly expanded using codified versions of \Cref{exp-process}, \Cref{D-process} and \Cref{beta-duality} until an equilibrium is reached. The key idea is then to represent each $\beta$ estimate as a polygon containing feasible $(\alpha, \beta)$ tuples, and then to take the intersection of all such polygons, which is easily achieved via standard boolean polytope operations. The resulting polygon may then be converted back into one or more \texttt{Hypothesis} objects representing a piecewise-linear bound on $\beta$. 

This machine-assisted approach uncovers a number of improvements to manually-derived $\beta$ bounds, particularly those bounds depending on multiple applications of exponent pair processes.
%, which are tedious to manually compute and previously unknown. 

\begin{figure}
\centering
\includegraphics[width=9cm]{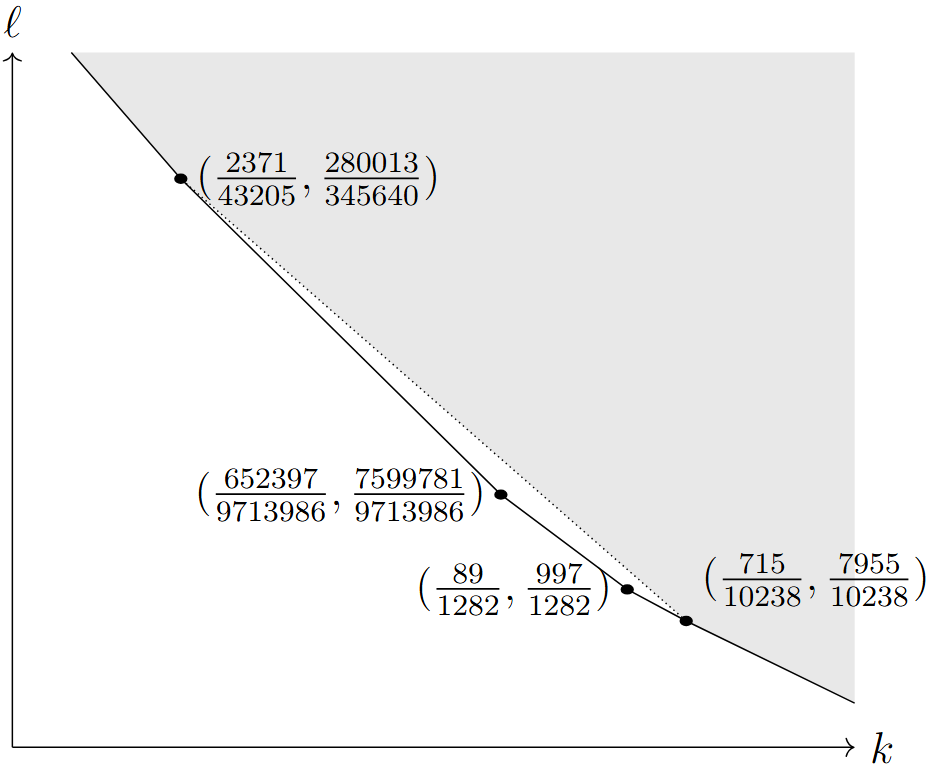}
\caption{Indicative plot of the new exponent pairs $(\frac{89}{1282}, \frac{997}{1282})$ and $(\frac{652397}{9713986}, \frac{7599781}{9713986})$ compared to previously known set of exponent pairs (shaded grey).}
\label{fig1}
\end{figure}

\begin{figure}
\centering
\includegraphics[width=9cm]{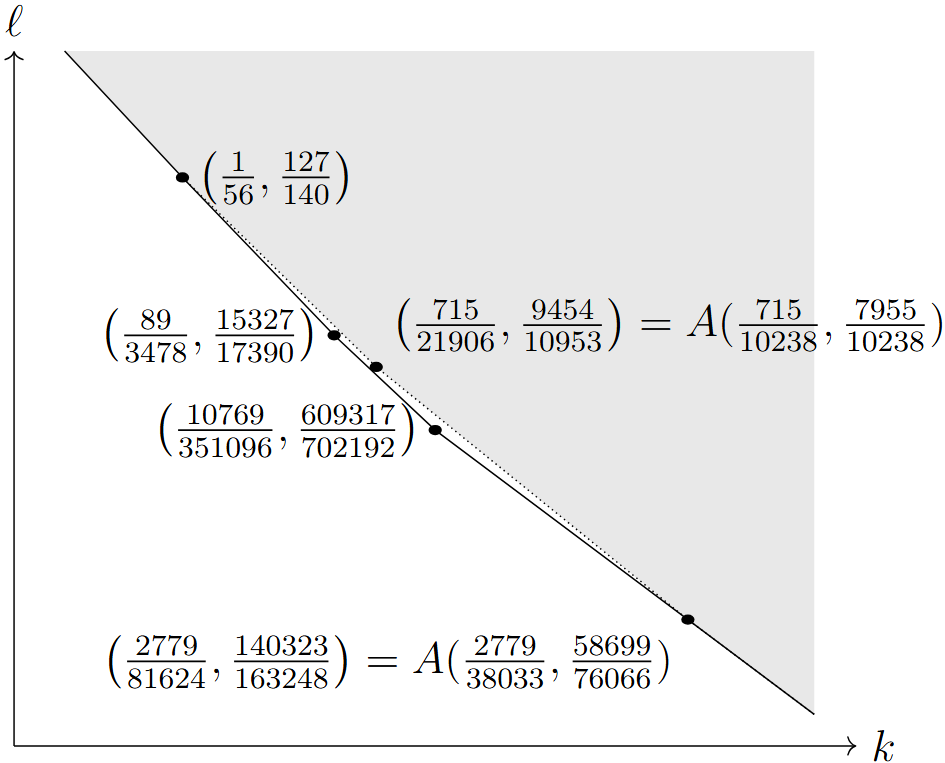}
\caption{Indicative plot of the new exponent pairs $\left(\frac{10769}{351096}, \frac{609317}{702192}\right)$ and $\left(\frac{89}{3478}, \frac{15327}{17390}\right)$ compared to previously known set of exponent pairs (shaded grey).}
\label{fig2}
\end{figure}
Lastly, for future reference, we also record a weaker exponent pair, needed for \Cref{hb-density2} below, that can be obtained from older results in the literature:

\begin{theorem}[Old exponent pair]\label{old-exp-pair} $(\frac{3}{40}, \frac{31}{40})$ is an exponent pair.
\end{theorem}

\begin{proof} The following proof was obtained by computer assistance, using the earliest possible result in the literature.

It was shown in \cite{watt_exponential_1989} that $W \coloneqq (\frac{89}{560}, \frac{1}{2} + \frac{89}{560})$ is an exponent pair.  Direct calculation shows that
$$ \left(\frac{3}{40}, \frac{31}{40}\right) = xy AW + (1-x)y ABAW + (1-y) W$$
with $x = 37081/40415$ and $y = 476897/493711$.  The claim now follows from \Cref{exp-process} and the evident convexity of exponent pairs.
\end{proof}

\section{Growth exponents for the Riemann zeta-function}

We record our notation for the growth exponents of the Riemann zeta-function:

\begin{definition}[Growth rate of $\zeta(s)$]\label{zeta-grow-def}  For any fixed $\sigma \in \R$, let $\mu(\sigma)$ denote the least possible (fixed) exponent for which one has the bound
    $$ |\zeta(\sigma+it)| \ll |t|^{\mu(\sigma)+o(1)}$$
    for all unbounded $t$.
\end{definition}

Non-asymptotically, $\mu(\sigma)$ is the least quantity such that for every $\eps>0$ there exists $C_\eps>0$ for which one has the bound
$$ |\zeta(\sigma+it)| \leq C_\eps |t|^{\mu(\sigma)+\eps}$$
whenever $|t| \geq C_\eps$.  Equivalently, one can define $\mu$ by the formula \eqref{mu-def}.

It is classical (see, e.g., \cite[Chapter 5]{titchmarsh_theory_1986}) that $\mu$ is convex, obeys the functional equation \eqref{mu-func}, and obeys the lower bound
$$
\mu(\sigma) \geq \max(0, 1/2-\sigma)
$$
for all $\sigma \in \R$.  The Lindel\"of hypothesis asserts that this lower bound is always sharp; this is known for $\sigma \leq 0$ or $\sigma \geq 1$.  Many partial results towards this conjecture, in the form of upper bounds for $\mu$, are known; we refer the reader to the ANTEDB for a current listing of all known bounds.

We have the following well-known relation between exponent pairs and $\mu$:

\begin{lemma}[Exponent pairs and $\mu$]\label{exp-pair-mu} If $(k,\ell)$ is an exponent pair, then $\mu(\ell-k) \leq k$.
\end{lemma}

\begin{proof} See e.g. \cite[(7.57)]{ivic}.
\end{proof}

Thus for instance the exponent pair conjecture implies the Lindel\"of conjecture.  Of course, one could combine this lemma with \Cref{new-exp-pair} or \Cref{old-exp-pair}.  We just record the latter combination below.

\begin{corollary}[Bound on $\mu$]\label{mu-bound}  We have $\mu(7/10) \leq 3/40$.
\end{corollary}

\section{Large value theorems}

It is well known that zero density estimates for $\zeta(s)$ (and other $L$-functions) are closely tied to large value theorems for Dirichlet polynomials; see for instance \cite{matomaki_teravainen_2024} for some recent discussion of this relationship.  We can formalise the latter by introducing the concept of a \emph{large value pattern}, as follows.  Call a sequence $a_n$ of complex numbers \emph{$1$-bounded} if $|a_n| \leq 1$ for all $n$ in the sequence, and call a set $W$ of real numbers \emph{$1$-separated} if $|t-t'| \geq 1$ for all distinct $t,t' \in W$.

\begin{definition}[Large value pattern]\label{large-pattern-def} A \emph{large value pattern} is a tuple $$(N, T, V, (a_n)_{n \in [N,2N]}, J, W),$$ where $N > 1$ and $T, V > 0$ are real numbers, $a_n$ is a $1$-bounded sequence on $[N,2N]$, $J$ is an interval of length $T$, and $W$ is a $1$-separated subset of $J$ such that
    \begin{equation}\label{V-large}
            \left|\sum_{n \in [N,2N]} a_n n^{-it} \right| \geq V
    \end{equation}
    for all $t \in W$.

    A \emph{zeta large value pattern} is a large value pattern in which $J = [T,2T]$ and $a_n = 1_I(n)$ for some interval $I \subset [N,2N]$, so that
    \begin{equation*}\label{V-large-zeta}
        \left|\sum_{n \in I} n^{-it} \right| \geq V
\end{equation*}
for all $t \in W$.
\end{definition}

\begin{remark}  In the literature, one also considers large value patterns in which the sequence $a_n$ is not controlled in an $\ell^\infty$ sense, but rather in an $\ell^2$ sense.  However, the $\ell^\infty$ control is sufficient for many applications of interest, particularly as we are permitting epsilon losses in the exponents (which in particular allows one to use the elementary divisor bound $\tau(n) \ll n^{o(1)}$ to ignore any losses relating to divisor sums).  Such $\ell^2$-variants of large value patterns are not currently studied in the ANTEDB, but could conceivably be addressed in a future expansion of the database.
\end{remark}

The key quantity of interest is then

\begin{definition}[Large value exponent]\label{lv-def} Let $1/2 \leq \sigma \leq 1$ and $\tau \geq 0$ be fixed. We define $\LV(\sigma,\tau) \in [-\infty,\infty)$ to be the least fixed quantity for which the following claim is true: whenever $(N,T,V,(a_n)_{n \in [N,2N]},J,W)$ is a large value pattern with $N>1$ unbounded, $T = N^{\tau+o(1)}$, and $V = N^{\sigma+o(1)}$, then
$$ |W| \ll N^{\LV(\sigma,\tau)+o(1)}.$$
We define $\LV_\zeta(\sigma,\tau)$ similarly, but where $(N,T,V,(a_n)_{n \in [N,2N]},J,W)$ is now required to be a zeta large value pattern.
\end{definition}

As with previous definitions, this exponent can be expressed non-asymptotically: $\LV(\sigma,\tau)$ is the infimum of all real numbers $\rho$ such that for every (fixed)  $\eps>0$ there exists $C, \delta>0$ such that if
$$(N,T,V,(a_n)_{n \in [N,2N]},J,W)$$
is a large value pattern with $N \geq C$ and $N^{\tau-\delta} \leq T \leq N^{\tau+\delta}$, $N^{\sigma-\delta} \leq V \leq N^{\sigma+\delta}$, then one has
$$ |W| \leq C N^{\rho+\eps}.$$
One can define $\LV_\zeta(\sigma,\tau)$ non-asymptotically in a similar fashion.
In the definition of $\LV(\sigma,\tau)$, one has the freedom to translate the interval $J$ by any shift $t_0$ by modulating the coefficients $a_n$ by $n^{-it_0}$.  As such, one can normalise $J$ to be $[0,T]$ without loss of generality for this quantity.  However, we do \emph{not} have this freedom for $\LV_\zeta(\sigma,\tau)$, since in that case the coefficients $a_n$ are required to equal $1$.

We begin with a discussion of the function $\LV(\sigma,\tau)$.
It is easy to see that $\LV(\sigma,0)=0$ for all $1/2 \leq\sigma \leq 1$.  We can also formalise the \emph{Huxley subdivision method} (see, e.g., \cite[Corollary 9.9]{ik}) as a functional inequality for $\LV$:

\begin{lemma}[Huxley subdivision]\label{hux-sub} If $1/2 \leq \sigma \leq 1$ and $0 \leq \tau \leq \tau'$, then
    $$ \LV(\sigma,\tau) \leq \LV(\sigma,\tau') \leq \LV(\sigma,\tau) + \tau'-\tau.$$
\end{lemma}

\begin{proof} This is clear since any interval of length $N^{\tau'+o(1)}$ can be subdivided into $N^{\tau'-\tau+o(1)}$ intervals of length $N^{\tau+o(1)}$.
\end{proof}

We stress that the subdivision inequality is only available for $\LV(\sigma,\tau)$ and not for $\LV_\zeta(\sigma,\tau)$, due to the aforementioned lack of freedom to translate the interval $J$ in the latter case.

\Cref{hux-sub} leads to a lower bound for $\LV(\sigma,\tau)$ in general:

\begin{lemma}[Lower bound]\label{lv-lower} For $\sigma=1/2$, one has $\LV(\sigma,\tau)=\tau$.  For $1/2 < \sigma \leq 1$, one has $\LV(\sigma,\tau) \geq \min(2-2\sigma,\tau)$.
\end{lemma}

\begin{proof}  For the $\sigma=1/2$ case, one can set $a_n = \pm 1$ to be random signs, and then a routine application of the Chernoff inequality shows that $\left|\sum_{n \in [N,2N]} a_n n^{-it} \right| = N^{1/2+o(1)}$ with high probability for any given $t$, which gives the claim. 
%$\LV(1/2,\tau)=\tau$.

For $1/2 < \sigma \leq 1$, we can modify this argument as follows.  By \Cref{hux-sub} it suffices to show that $\LV(\sigma,2-2\sigma) \geq 2-2\sigma$.  Divide the interval $[N,2N]$ into $\asymp N^{2-2\sigma}$ intervals  $I$ of length $\asymp N^{2\sigma-1}$.  On each interval $I$, we choose $a_n$ to equal some randomly chosen sign $\epsilon_I \in \{-1,+1\}$, with the $\epsilon_I$ chosen independently in $I$.  If $t = o(N^{2-2\sigma})$, then $\sum_{n \in I} a_n n^{-it}$ is equal to $\epsilon_I$ times a deterministic quantity $c_{t,I}$ of magnitude $\asymp N^{2\sigma-1}$ (the point being that the phase $t \log n$ is close to constant in this range).  By the Chernoff bound, we thus see that for any such $t$, $\sum_{n \in [N,2N]} a_n n^{it}$ will have size $\gg N^{(2\sigma-1) + (2-2\sigma)/2} = N^\sigma$ with probability $\gg 1$. By linearity of expectation, we thus see that with positive probability, a $\gg 1$ fraction of integers $t$ with $t = o(N^{2-2\sigma})$ will have this property, giving the claim.
\end{proof}

The \emph{Montgomery conjecture} essentially asserts that this lower bound is sharp:

\begin{conjecture}[Montgomery conjecture]\label{montgomery-conj'}
One has $\LV(\sigma, \tau) = \min(2 - 2 \sigma,\tau)$ for all $1/2 < \sigma \leq 1$ and $\tau \geq 0$.  Equivalently (thanks to \Cref{lv-lower} and the trivial bound $\LV(\sigma,\tau) \leq \tau$), one has $\LV(\sigma,\tau) \leq 2-2\sigma$ for all $1/2 < \sigma \leq 1$ and $\tau \geq 0$.
\end{conjecture}

One can phrase this conjecture in the equivalent form
\begin{equation}\label{montgomery-alt}
     \sum_{t \in W} \bigg|\sum_{n \in [N,2N]} a_n n^{-it}\bigg|^2 \ll N^{1+o(1)} (T + N) \bigg(\sup_{n \in [N,2N]} |a_n|\bigg)^2
\end{equation}
for any unbounded $N$, any coefficients $a_n$, and any $1$-separated subset $W$ of an interval $J$ of length $T = N^{O(1)}$. Indeed, \eqref{montgomery-alt} can be easily seen to imply \Cref{montgomery-conj'} using Markov's inequality, while, the converse implication follows from a standard dyadic decomposition of $W$ into $O(\log N) = O(N^{o(1)})$ components depending on the dyadic size of $|\sum_{n \in [N,2N]} a_n n^{-it}|$. The conjecture \eqref{montgomery-alt} appears (in slightly different notation) in \cite[(1.6)]{bourgain_montgomery_1991} and is attributed to Montgomery as a correction of his original conjecture in \cite{montgomery_topics_1971}, which was shown in the former paper to be too strong to be completely correct.  In that paper it was also shown that the conjecture \eqref{montgomery-alt} implies the Kakeya conjecture in geometric measure theory.

From the trivial upper bound $\LV(\sigma,\tau) \leq \tau$ we see that the Montgomery conjecture is true for $\tau \leq 2-2\sigma$.  The range of validity of the conjecture can be extended.  For instance, we have

\begin{theorem}[$L^2$ mean value theorem]\label{l2-mvt} For any fixed $1/2 \leq \sigma \leq 1$ and $\tau\geq 0$ one has
    $$ \LV(\sigma,\tau) \leq \max( 2-2\sigma, 1 + \tau - 2 \sigma).$$
In particular, the Montgomery conjecture holds for $\tau \leq 1$.
\end{theorem}

\begin{proof}
Let $(N,T,V,(a_n)_{n \in [N,2N]},J,W)$ be a large value pattern with $T = N^{\tau + o(1)}$, $V = N^{\sigma + o(1)}$. Applying \cite[Theorem~9.4]{ik} (with $N$, $T$ replaced with $2N$, $2T$ respectively and taking $a_n = 0$ for $n < N$) one has
\[
|W|V^2 \le \sum_{t \in W} \bigg|\sum_{N \le n \le 2N} a_n n^{-it}\bigg|^2 \ll N^{1 + o(1)}(T + N).
\]
The result follows from comparing exponents.
\end{proof}

We record some further results of this type in the literature, again valid for any fixed  $1/2 \leq \sigma \leq 1$ and $\tau\geq 0$:
\begin{itemize}
\item[(i)]  The Huxley large values theorem \cite[Equation~(2.9)]{Huxley}, in our notation, gives the upper bound
\begin{equation}\label{huxley-lvt}
     \LV(\sigma,\tau) \leq \max( 2-2\sigma, 4 + \tau - 6 \sigma).
\end{equation}
In particular, one has the Montgomery conjecture for $\tau \leq 4 \sigma - 2$.
\item[(ii)] The Heath-Brown large values theorem \cite[p. 226]{heathbrown_zero_1979} similarly gives
\begin{equation}\label{hb-opt}
    \LV(\sigma,\tau) \leq \max( 2-2\sigma, 10 + \tau - 13 \sigma).
\end{equation}
In particular, the Montgomery conjecture holds for $\tau \leq 11 \sigma - 8$.
\item[(iii)] The Jutila large values theorem \cite[(1.4)]{jutila_zero_density_1977} gives, for any positive integer $k$, the bound
\begin{equation}\label{jutila-lvt}
    \LV(\sigma,\tau) \leq \max(2-2\sigma, \tau + (4-2/k) - (6-2/k)\sigma, \tau + (6-8\sigma)k).
\end{equation}
In particular, the Montgomery conjecture holds for
$$ \tau \leq \min( (4-2/k)\sigma - (2-2/k), (8k-2)\sigma - 6k + 2).$$
\end{itemize}

There are several further large value theorems in the literature, though not necessarily of a type that further extends the range of validity of the Montgomery conjecture; for instance, the recent result \cite[Theorem~1.1]{guth-maynard} of Guth and Maynard, in our notation, gives the bound
\begin{equation}\label{guth-maynard-lvt}
    \LV(\sigma,\tau) \leq \max(2-2\sigma, 18/5 - 4 \sigma, \tau + 12/5 - 4\sigma).
\end{equation}
We refer the reader to the ANTEDB for a current listing of all known large value theorems.

We now record a large value theorem of Bourgain \cite{bourgain_large_2000}, stated here in more generality than was explicitly stated in that paper, and restated in our notation for compatibility with the other results listed here.

\begin{theorem}[Bourgain large values theorem]\label{bourgain-lvt}\cite{bourgain_large_2000} Let $1/2 < \sigma < 1$ and $\tau > 0$, and let $\rho := \LV(\sigma,\tau)$.  Let $\alpha_1, \alpha_2 \geq 0$ be real numbers.
    \begin{itemize}
        \item[(i)]  Either
    \begin{equation}\label{rho1}
     \rho \leq \max( \alpha_2 + 2 - 2 \sigma, -\alpha_2 + 2\tau+4-8\sigma, -2\alpha_1 + \tau + 12 - 16 \sigma)
    \end{equation}
    or else there exists $s \geq 0$ such that
    \begin{equation}\label{rs}
    \begin{split}
         &\frac{1}{2}\max(\rho+2, 2\rho+1, 5\rho/4 + \tau/2 + 1) + \frac{1}{2}\max(s+2, 2s+1, 5s/4 + \tau/2 + 1) \geq \\
            &\qquad\max( -2\alpha_1 + 2\sigma + s + \rho, -\alpha_1 - \alpha_2/2 + 2\sigma + s/2 + 3\rho/2).
    \end{split}
    \end{equation}
        \item[(ii)] \cite[Lemma 4.60]{bourgain_large_2000} If we additionally assume $\rho \leq \min(1, 4-2\tau)$, then
        \begin{align*}
            \rho &\leq \max( \alpha_2 + 2 - 2 \sigma, \alpha_1+\alpha_2/2 + 2-2\sigma, -\alpha_2 + 2\tau+4-8\sigma, \\
            &\quad -2\alpha_1 + \tau + 12 - 16 \sigma, 4\alpha_1 + 2+\max(1,2\tau-2)-4\sigma).
        \end{align*}
\end{itemize}
    \end{theorem}

    \begin{proof}  We begin with (i). By Definition \ref{lv-def}, we can find a large value pattern $(N,T,V,(a_n)_{n \in [N,2N]},J,R)$ with $N>1$ unbounded, $N \geq 1$, $T = N^{\tau+o(1)}$, $|R| = N^{\rho+o(1)}$, $V = N^{\sigma+o(1)}$; we use $R$ here instead of $W$ to be consistent with the notation from \cite{bourgain_large_2000}.  Now set $\delta_1 := N^{-\alpha_1}$, $\delta_2 := N^{-\alpha_2}$.  From \cite[(4.41), (4.42)]{bourgain_large_2000}, one has the inequality
        $$ |R| \leq |R^{(1)}| + |R^{(2)}|$$
        for certain sets $R^{(1)}$ and $R^{(2)}$ with the former set obeying the bound
    $$ |R^{(1)}| \ll \delta_2^{-1} N^2 V^{-2} + \delta_2 T^2 N^4 V^{-8} + \delta_1^2 T N^{12} V^{-16}.$$
    Hence, we either have
    $$ |R| \ll \delta_2^{-1} N^2 V^{-2} + \delta_2 T^2 N^4 V^{-8} + \delta_1^2 T N^{12} V^{-16}$$
    which implies \eqref{rho1}, or else
    \begin{equation}\label{rr}
        |R| \ll |R^{(2)}|.
    \end{equation}
    Henceforth we assume that \eqref{rr} holds. From \cite[(4.53), (4.54)]{bourgain_large_2000} we may  bound
    \begin{equation}\label{Exp-1}
         T^{-\eps} \delta' (\delta'')^2 V^2 |S| |R^{(2)}| + T^{-\eps} \delta_1 V^2 |S|^{1/2} \sum_\alpha |R_\alpha|^{3/2}
    \end{equation}
  thus
    \begin{equation*}\label{Exp-2}
        \ll T^\eps S(N,R^{(2)})^{1/2} S(N,S)^{1/2}
    \end{equation*}
    for arbitrarily small fixed $\eps$, some $\delta',\delta''>0$ with $\delta' > T^{-\eps} (\delta_1/\delta'')^2$ (see \cite[(4.37)]{bourgain_large_2000}), some set $S$ (which will be non-empty by \cite[(4.47)]{bourgain_large_2000}), and some sets $R_\alpha$ defined in \cite[(4.39)]{bourgain_large_2000}, where the double zeta sums $S(N,W)$ are defined in \eqref{snw}.  Applying \cite[Theorem 1]{heathbrown_large_1979} or \cite[Lemma 11.5]{ivic}, the latter expression is bounded by
    $$ \ll T^\eps (|R|N^2 + |R|^2 N + |R|^{5/4} T^{1/2} N)^{1/2} (|S| N^2 + |S|^2 N + |S|^{5/4} T^{1/2} N)^{1/2}.$$
    Meanwhile, from \cite[(4.57)]{bourgain_large_2000}, the expression \eqref{Exp-1} is bounded from below by
    $$ \gg T^{-2\eps} (\delta_1^2 V^2 |S| |R| + \delta_1 \delta_2^{1/2} V^2 |S|^{1/2} |R|^{3/2}).$$
    After passing to a subsequence, we can ensure that $|S| = N^{s+o(1)}$ for some $s > 0$.
    Combining these bounds and writing all expressions as powers of $N$, we obtain the claim (i) (after sending $\eps \to 0$).

    Now we prove (ii).  With $\rho \leq \min(1,4-2\tau)$, $5\rho/4+\tau/2+1$ and $2\rho+1$ are both bounded by $\rho+2$, hence
    $$ \max(\rho+2, 2\rho+1, 5\rho/4 + \tau/2 + 1) = \rho+2.$$
    Furthermore, $5s/4+\tau+1$ is a convex combination of $s+2$ and $2s + 2\tau-2$, hence
    $$\max(s+2, 2s+1, 5s/4 + \tau/2 + 1) \leq \max(s+2, 2s + \max(1,2\tau-2)).$$
    Thus \eqref{rs} simplifies to
    \begin{align*}
        &(\rho+2)/2 + \max(s+2, 2s+\max(1,2\tau-2))/2 \\
        &\quad \geq
        \max( -2\alpha_1 + 2\sigma + s + \rho, -\alpha_1 - \alpha_2/2 + 2\sigma + s/2 + 3\rho/2).
    \end{align*}
    Thus either
    $$(\rho+2)/2 + (s+2)/2 \geq -\alpha_1 - \alpha_2/2 + 2\sigma + s/2 + 3\rho/2$$
    or
    $$(\rho+2)/2 + (2s+\max(1,2\tau-2))/2 \geq  -2\alpha_1 + 2\sigma + s + \rho.$$
    In both cases we may eliminate $s$ and solve for $\rho$ to obtain
    $$ \rho \leq \alpha_1 + \alpha_2/2 + 2 - 2 \sigma $$
    or
    $$ \rho \leq 4\alpha_1 + 2 + \max(1,2\tau-2) - 4 \sigma,$$
    giving the claim.
    \end{proof}

The standard trick of increasing the range of applicability of large value theorems by raising Dirichlet series to powers can also be cleanly encoded in our framework:

\begin{lemma}[Raising to a power]\label{power-lemma} For any fixed $1/2 \leq \sigma \leq 1$, $\tau \geq 0$, and natural number $k$, one has
    $$ \LV(\sigma, k\tau) \leq k \LV(\sigma, \tau).$$
\end{lemma}

\begin{proof} Let  $(N,T,V,(a_n)_{n \in [N,2N]},J,W)$ be a large value pattern with $T = N^{k\tau+o(1)}$ and $V = N^{\sigma+o(1)}$. Raising \eqref{V-large} to the $k^{\mathrm{th}}$ power, we conclude that
$$ \bigg|\sum_{n \in [N^k,2^kN^k]} b_n n^{-it} \bigg| \geq V^k$$
for all $t \in W$, where $b_n$ is the Dirichlet convolution of $k$ copies of $a_n$, and thus is bounded by $N^{o(1)}$ thanks to divisor bounds.  Subdividing $[N^k, 2^k N^k]$ into $k$ intervals of the form $[N',2N']$ for $N' \asymp N^k$ and applying Definition \ref{lv-def} (with $N, T, V$ replaced by $N', T, V^k/N^{o(1)}$) we conclude that
$$ |W| \ll N^{k \LV(\sigma,\tau) + o(1)}$$
and the claim then follows.
\end{proof}

\subsection{Large value theorems for the zeta-function}

Now we discuss the related quantity $\LV_\zeta(\sigma,\tau)$.  We trivially have the pointwise bound
\begin{equation}\label{lvz-lv}
\LV_\zeta(\sigma,\tau) \leq \LV(\sigma,\tau).
\end{equation}
We will primarily be interested in this function in the regime $\tau \geq 2$, as this is the region relevant to the Riemann--Siegel formula (or other approximate functional equations) for the Riemann zeta-function, and in any event the $1 < \tau < 2$ range can be related to the $\tau>2$ range using \Cref{add-bound}(iii) below.

Additional bounds are available through pointwise or moment bounds on the zeta-function, as well as from the (approximate) functional equation:

\begin{lemma}[Additional bounds]\label{add-bound}\
\begin{itemize}
\item[(i)] If $\tau>0$ and $1/2 \leq \sigma_0 \leq 1$ are fixed, then $\LV_\zeta(\sigma,\tau) = -\infty$ whenever $\sigma > \sigma_0 + \tau \mu(\sigma_0)$.
\item[(ii)]  If one has a moment estimate of the form
\begin{equation}\label{moment}
    \int_T^{2T} |\zeta(\sigma_0+it)|^A \ll T^{M+o(1)}
\end{equation}
for some fixed $1/2 \leq \sigma_0 \leq 1$ and $A \geq 1$, and all unbounded $T$, then $\LV_\zeta(\sigma,\tau) \leq \tau M - A(\sigma-\sigma_0)$ for any fixed $1/2 \leq \sigma \leq 1$ and $\tau \geq 2$.
\item[(iii)] (Reflection) For fixed $1/2 \leq \sigma \leq 1$ and $\tau > 1$, one has
\begin{align*}
    & \sup_{\sigma \leq \sigma' \leq 1} \LV_\zeta\left(\frac{1}{2} + \frac{1}{\tau-1} \left(\sigma'-\frac{1}{2}\right), \frac{\tau}{\tau-1}\right) + \frac{1}{\tau-1} (\sigma'-\sigma) \\
    &\quad = \frac{1}{\tau-1} \sup_{\sigma \leq \sigma' \leq 1} (\LV_\zeta(\sigma',\tau) + \sigma'-\sigma).
\end{align*}
\end{itemize}
\end{lemma}

\begin{proof}  For (i), observe from \cite[(8.13)]{ivic} that
$$ \sum_{n \in I} \frac{1}{n^{\sigma_0+it}} \ll |t|^{\mu(\sigma_0) + o(1)}$$
    for unbounded $N$, if $I \subset [N,2N]$ and $|t| = N^{\tau+o(1)}$.  By partial summation this gives
$$ \sum_{n \in I} n^{-it} \ll N^{\sigma_0} |t|^{\mu(\sigma_0) + o(1)} = N^{\sigma_0 + \tau \mu(\sigma_0) + o(1)}.$$
The claim follows.

For (ii), let $(N,T,V,(a_n)_{n \in [N,2N]},J,W)$ be a zeta large value pattern with
    $N$ unbounded, $T= N^{\tau+o(1)}$, and $V = N^{\sigma+o(1)}$, then
$$ \bigg|\sum_{n \in J} n^{-it}\bigg| \gg N^{\sigma+o(1)}$$
for all $t \in W$.
By standard Fourier analysis (or from the Perron formula and contour shifting), this gives
$$ \int_{T/2}^{3T} |\zeta(\sigma_0+it')|\ \frac{dt'}{1+|t'-t|} \gg N^{\sigma - \sigma_0 + o(1)}$$
and hence by H\"older
$$ \int_{T/2}^{3T} |\zeta(\sigma_0+it')|^A\ \frac{dt'}{1+|t'-t|} \gg N^{A(\sigma - \sigma_0) + o(1)}$$
so on summing in $t$
$$ \int_{T/2}^{3T} |\zeta(\sigma_0+it)|^A\ dt \gg |W| N^{A(\sigma - \sigma_0) + o(1)}.$$
By hypothesis, the left-hand side is $\ll T^{M+o(1)}$.  Since $T = N^{\alpha+o(1)}$, we obtain
$$ |W| \ll N^{\tau M - A(\sigma-\sigma_0)},$$
giving the claim.

Now we turn to (iii).  By symmetry it suffices to prove the upper bound.  Actually it suffices to show
    $$ \LV_\zeta\left(\frac{1}{2} + \frac{1}{\tau-1} (\sigma-\frac{1}{2}), \frac{\tau}{\tau-1}\right) \leq \frac{1}{\tau-1} \sup_{\sigma \leq \sigma' \leq 1} (\LV_\zeta(\sigma',\tau) + \sigma'-\sigma)$$
    as this easily implies the general upper bound.

    Let $(N,T,V,(a_n)_{n \in [N,2N]},J,W)$ be a zeta large value pattern with
    $N$ unbounded, $T= N^{\frac{\tau}{\tau-1}+o(1)}$, and $V = N^{\frac{1}{2} + \frac{1}{\tau-1} (\sigma-\frac{1}{2})+o(1)}$. By definition, it suffices to show the bound
    \begin{equation}\label{r-targ}
        |W| \ll N^{\frac{1}{\tau-1} (\LV_\zeta(\sigma',\tau) + \sigma'-\sigma)+o(1)},
    \end{equation}
    for some $\sigma \leq \sigma' \leq 1$.
    By definition, $a_n = 1_I(n)$. By a Fourier expansion of $(n/N)^{1/2}$ in $\log n$, we can bound
    $$ \bigg|\sum_{n \in I} n^{-it_r}\bigg| \ll_A N^{1/2} \int_\R \bigg|\sum_{n \in I} n^{-1/2-it}\bigg|\ \frac{dt}{(1 + |t-t_r|)^{-A}}$$
    and hence by the pigeonhole principle, we can find $t' = t + O(N^{o(1)})$ for each $t \in W$ such that
    $$ \bigg|\sum_{n \in I} n^{-1/2-it'}\bigg| \gg N^{-1/2-o(1)} V$$
    for $t \in W$.  By refining $W$ by $N^{o(1)}$ if necessary, we may assume that the $t'$ are $1$-separated.

    Now we use the approximate functional equation
    \begin{align*}
        \zeta(1/2+it') = &\sum_{n \leq x} n^{-1/2-it'} + \chi(1/2+it') \sum_{m \leq t' / 2\pi x} m^{-1/2+it'} \\
        &\qquad + O(N^{-1/2}) + O((T/N)^{-1/2})
    \end{align*}
    for $x \asymp N$; see \cite[Theorem 4.1]{ivic}.  Applying this to the two endpoints of $I$ and subtracting, we conclude that
    $$ \sum_{n \in I} n^{-1/2-it'} =\chi(1/2+it') \sum_{m \in J_{t'}} m^{-1/2+it'} + O(N^{-1/2}) + O((T/N)^{-1/2})$$
    where $J_{t'} := \{ m: t' / 2\pi m \in I \}$. Since $\chi(1/2+it')$ has magnitude one, we conclude that
    $$ \bigg|\sum_{m \in J_r} m^{-1/2-it'}\bigg| \gg N^{-1/2-o(1)} V.$$
    Writing $M := T/N = N^{\frac{1}{\tau-1}+o(1)}$, we see that $J_r \subset [M/10, 10M]$ and
    $$ \bigg|\sum_{m \in J_r} (M/m)^{1/2} m^{-it'}\bigg| \gg M^{1/2} N^{-1/2-o(1)} V = M^{\sigma+o(1)}.$$
    Performing a Fourier expansion of $(M/m)^{1/2} 1_{J_r}(m)$ (smoothed out at scale $O(1)$) in $\log m$, we can bound
    $$ \bigg|\sum_{m \in J_r} (M/m)^{1/2} m^{-it'}\bigg| \ll \int_{T/10}^{10T} \bigg|\sum_{m \in [M/10,10M]} m^{-it_1}\bigg|\ \frac{dt_1}{1 + |t_1-t'|} + T^{-10}$$
    and hence
    $$ \int_{T/10}^{10T} \bigg|\sum_{m \in [M/10,10M]} m^{-it_1}\bigg|\ \frac{dt_1}{1 + |t_1-t'|} \gg M^{\sigma+o(1)}.$$
    If we let $E$ denote the set of $t_1 \in [T/10, 10T]$ for which $|\sum_{m \in [M/10,10M]} m^{-it_1}| \geq M^{\sigma-o(1)}$ for a suitably chosen $o(1)$ error, then we have
    $$ \int_E \bigg|\sum_{m \in [M/10,10M]} m^{-it_1}\bigg|\ \frac{dt_1}{1 + |t_1-t'|}  \gg M^{\sigma+o(1)}.$$
    Summing in $t'$, we obtain
    $$ \int_E \bigg|\sum_{m \in [M/10,10M]} m^{-it_1}\bigg|\ dt_1  \gg M^{\sigma+o(1)} R$$
    and so by dyadic pigeonholing we can find $M^{\sigma-o(1)} \ll V'' \ll M$ and a $1$-separated subset $W''$ of $E$ such that
    $$ \bigg|\sum_{m \in [M/10,10M]} m^{-it''}\bigg|\ dt \asymp V''$$
    for all $t'' \in W''$, and
    $$ V'' |W''| \gg M^{\sigma+o(1)} |W|.$$
    By passing to a subsequence we may assume that $V'' = M^{\sigma'+o(1)}$ for some $\sigma \leq \sigma' \leq 1$. Partitioning $[M/10,10M]$ into a bounded number of intervals each of which lies in a dyadic range $[M',2M']$ for some $M' \asymp M$, and \Cref{lv-def}, we have
    $$ |W''|  \ll M^{\LV_\zeta(\sigma',\tau)+o(1)}$$
    and \eqref{r-targ} follows.
\end{proof}

\begin{remark} The arguments for (ii) can be reversed, in that the optimal exponent $M$ in the indicated moment bound can be shown to be
$$ M = \sup_{\tau \geq 2; \sigma \geq 1/2} (A(\sigma-\sigma_0) + \LV_\zeta(\sigma,\tau))/\tau.$$
We refer the reader to the ANTEDB for a proof.  Thus the moment exponents for the zeta-function are in some sense dual to the function $LV_\zeta$, somewhat analogously to the relationship between exponent pairs and the function $\beta$ in \Cref{beta-duality}.
\end{remark}

\begin{remark}
    We note that in practice, bounds for $\LV_\zeta(\sigma',\tau)+\sigma'$ are monotone decreasing\footnote{This reflects the fact that large value theorems usually relate to $p^{\mathrm{th}}$ moment bounds for $p \geq 1$ (e.g., $p = 2, 4, 6, 12$) rather than for $0 < p < 1$.} in $\sigma'$, so the reflection property in Lemma \ref{add-bound}(iii) morally simplifies\footnote{One could redefine $\LV_\zeta$ to use smooth cutoffs in the $n$ variable rather than rough cutoffs $1_I(n)$, whence one could obtain the analogue of \eqref{lvz-reflect} rigorously, but we do not do so here.} to
\begin{equation}\label{lvz-reflect}
    \LV_\zeta\left(\frac{1}{2} + \frac{1}{\tau-1} \left(\sigma-\frac{1}{2}\right), \frac{\tau}{\tau-1}\right) = \frac{1}{\tau-1} \LV_\zeta(\sigma,\tau).
\end{equation}
\end{remark}

The well-known twelfth moment bound of Heath--Brown \cite{heathbrown_twelfth_1978} gives \eqref{moment} with $\sigma_0=1/2$, $A=12$, and $M=2$.  Combining this with \Cref{add-bound}(ii), we see that
\begin{equation}\label{twelfth-bound} \LV_\zeta(\sigma,\tau) \leq 2\tau - 12 (\sigma-1/2)
\end{equation}
any fixed $1/2 \leq \sigma \leq 1$ and $\tau \geq 2$. In a similar vein, using the specific bound on $\mu$ in \Cref{mu-bound}, we concude that
\begin{equation}\label{lvz-340}
    \LV_\zeta(\sigma,\tau) = -\infty
\end{equation}
whenever $\tau>0$ and $1/2 \leq \sigma \leq 1$ are such that $\sigma > \frac{7}{10} + \frac{3}{40} \tau$.

There are several further results in the literature that can be viewed as bounds on $\LV_\zeta$, or relationships between $\LV_\zeta$ and other quantities such as exponent pairs; we refer  to the ANTEDB for a current listing of all known results of this type.

\section{Zero density estimates}

Our convention for zero density exponents will be as follows.

\begin{definition}[Zero density exponents]\label{zero-def}  For $\sigma \in \R$ and $T>0$, let $N(\sigma,T)$ denote the number of zeroes $\rho$ of the Riemann zeta-function with $\mathrm{Re}(\rho) \geq \sigma$ and $|\mathrm{Im}(\rho)| \leq T$.

    If $1/2 \leq \sigma < 1$ is fixed, we define the zero density exponent $\A(\sigma) \in [-\infty,\infty)$ to be the infimum of all (fixed) exponents $A$ for which one has
        $$ N(\sigma-\delta,T) \ll T^{A (1-\sigma)+o(1)}$$
    whenever $T$ is unbounded and $\delta>0$ is infinitesimal.
\end{definition}

In non-asymptotic terms: $\A(\sigma)$ is the infimum of all $A$ such that for every $\eps>0$ there exists $C, \delta > 0$ such that
$$ N(\sigma-\delta, T) \leq C T^{A(1-\sigma)+\eps}$$
whenever $T \geq C$.  Note that in other literature (e.g., \cite{ivic}), the shift by $\delta$ is not present, but this only affects the value of $\A(\sigma)$ at discontinuities of the function.  As $N(\sigma,T)$ is obviously monotone non-increasing in $\sigma$, we certainly have the bound
$$ N(\sigma,T) \ll T^{\A(\sigma)(1-\sigma)+o(1)}$$
for any fixed $1/2 \leq \sigma < 1$ and unbounded $T$.

The key conjecture for this exponent is

\begin{conjecture}[Density hypothesis]\label{density-hypothesis}  One has $\A(\sigma) \leq 2$ for all $1/2 \leq \sigma < 1$.
\end{conjecture}

Of course, the Riemann hypothesis gives the even stronger assertion $\A(\sigma)=-\infty$ for all $1/2 < \sigma < 1$. We refer the reader to \cite[Chapter 9]{titchmarsh_theory_1986} for further discussion of the density hypothesis.  From the Riemann--von Mangoldt formula one sees that
\begin{equation}\label{adens}
    \A(1/2) = 2
\end{equation}
so that the bound of $2$ in the density hypothesis cannot be reduced.

It is well known (see, e.g., \cite[Chapter 10]{ik}) that zero density estimates can be obtained from large value theorems through the device of zero detecting polynomials.  We formalise this principle as follows:

\begin{lemma}[Zero density from large values]\label{zero-from-large}  Let $1/2 < \sigma < 1$.  Then
    $$ \A(\sigma)(1-\sigma) \leq \max( \sup_{\tau \geq 2} \LV_\zeta(\sigma,\tau)/\tau, \limsup_{\tau \to \infty} \LV(\sigma,\tau)/\tau ).$$
    \end{lemma}

\begin{proof}
    Write the right-hand side as $B$, then $B \geq 0$ (from \Cref{lv-lower}) and we have
    \begin{equation}\label{lvz-bound}
        \LV_\zeta(\sigma,\tau) \leq B \tau
    \end{equation}
    for all $\tau \geq 1$, and
    \begin{equation}\label{lv-bound}
        \LV(\sigma,\tau) \leq (B+\eps) \tau
    \end{equation}
    whenever $\eps>0$ and $\tau$ is sufficiently large depending on $\eps$ (and $\sigma$).  It would suffice to show, for any $\eps>0$, that $N(\sigma-o(1),T) \ll T^{B+O(\eps)+o(1)}$ as $T \to \infty$.

    By dyadic decomposition, it suffices to show for large $T$ that the number of zeroes with real part at least $\sigma-o(1)$ and imaginary part in  $[T,2T]$ is $\ll T^{B+O(\eps)+o(1)}$.  From the Riemann--von Mangoldt formula, there are only $O(\log T)$ zeroes whose imaginary part is within $O(1)$ of a specified ordinate $t \in [T,2T]$, so it suffices to show that given some zeroes $\sigma_r + i t_r$, $r=1,\dots,R$ with $\sigma-o(1) \leq \sigma_r < 1$ and $t_r \in [T,2T]$ $1$-separated, that $R \ll T^{B+O(\eps)+o(1)}$.

    Suppose that one has a zero $\sigma_r+i t_r$ of this form.  Then by a standard approximation to the zeta-function \cite[Theorem 1.8]{ivic}, one has
    $$ \sum_{n \leq T} \frac{1}{n^{\sigma_r+it_r}} \ll T^{-1/2}.$$
    Let $0 < \delta_1 < \eps$ be a small quantity (independent of $T$) to be chosen later, and let $0 < \delta_2 < \delta_1$ be sufficiently small depending on $\delta_1,\delta_2$.  By the triangle inequality, and refining the sequence $t_r$ by a factor of at most $2$, we either have
    $$ \bigg|\sum_{T^{\delta_1} \leq n \leq T} \frac{1}{n^{\sigma_r+it_r}} \bigg| \gg T^{-\delta_2}$$
    for all $r$, or
    \begin{equation}\label{td}
     \sum_{n \leq T^{\delta_1}} \frac{1}{n^{\sigma_r+it_r}} \ll T^{-\delta_2}
    \end{equation}
    for all $r$.

    Suppose we are in the former (``Type I'') case, we perform a smooth partition of unity, and conclude that
    $$ \bigg|\sum_{T^{\delta_1} \leq n \leq T} \frac{\psi(n/N)}{n^{\sigma_r+it_r}} \bigg| \gg T^{-\delta_2 - o(1)}$$
    for some fixed bump function $\psi$ supported on $[1/2,1]$, and some $T^{\delta_1} \ll N \ll T$.

    We divide into several cases depending on the size of $N$.  First suppose that $N \ll T^{1/2}$.  The variable $n$ is restricted to the interval $I := [\max(N/2, T^{\delta_1}), N]$.  We have
    $$ \bigg|\sum_{n \in I} \psi(n/N) (n/N)^{-\sigma_r} n^{-it_r} \bigg| \gg N^\sigma T^{-\delta_2 - o(1)}.$$
    Performing a Fourier expansion of $\psi(n/N) (n/N)^{-\sigma_r}$ in $\log n$ and using the triangle inequality, we can bound
    $$ \sum_{n \in I} \psi(n/N) (n/N)^{-\sigma_r} n^{-it_r}  \ll_A \int_\R \bigg| \sum_{n \in I} \frac{1}{n^{it}} \bigg| (1+|t-t_r|)^{-A}\ dt$$
    for any $A>0$, so by the triangle inequality we conclude that
    $$ \bigg|\sum_{n \in I} n^{-it'_r} \bigg| \gg N^\sigma T^{-\delta_2 - o(1)}$$
    for some $t'_r = t_r + O(T^{o(1)})$.  By refining the $t_r$ by a factor of $T^{o(1)}$ if necessary, we may assume that the $t'_r$ are $1$-separated, and by passing to a subsequence we may assume that $T = N^{\tau+o(1)}$ for some $2 \leq \tau \leq 1/\delta_1$, then
    we conclude that
    $$ \bigg| \sum_{n \in I} \frac{1}{n^{it'_r}} \bigg| \gg N^{\sigma-\delta_2/\delta_1+o(1)}$$
    for all remaining $r$.  By Definition \ref{lv-def} we then have (for $\delta_2$ small enough)
    $$ R \ll N^{\LV_\zeta(\sigma,\tau) + \eps + o(1)} \ll T^{\LV_\zeta(\sigma,\tau)/\tau + \eps + o(1)}$$
    and the claim follows in this case from \eqref{lvz-bound}.

    In the case $N \asymp T$, a standard application of the Euler--Maclaurin formula 
%(see e.g., \cite[(2.1.2)]{titchmarsh_theory_1986}) TST: we may not need this again. 
yields
    $$ \sum_{T^{\delta_1} \leq n \leq T} \frac{\psi(n/N)}{n^{\sigma_r+it_r}} \ll T^{-\sigma_r}$$
    which leads to a contradiction.  So the only remaining case is when $T^{1/2} \ll N \ll o(T)$.  Here we can ignore the cutoffs on $n$ and obtain
    $$ \left| \sum_{n} \psi(n/N) (n/N)^{-\sigma_r} n^{-it_r}\right| \gg N^{\sigma} T^{-\delta_2-o(1)}.$$
    Applying the van der Corput $B$-process (see, e.g., \cite[\S 8.3]{ik}) or the approximate functional equation we have
\begin{align*}
    &\sum_{n} \psi(n/N) (n/N)^{-\sigma_r} n^{-it_r} = e\left(\frac{t_r}{2\pi} \log \frac{t_r}{2\pi} - \frac{t_r}{2\pi} + \frac{1}{8}\right)\\
    &\quad \times \sum_{m} \psi(2\pi t_r/mN) (2\pi t_r/mN)^{-\sigma_r} m^{it_r} (2\pi m^2/t_r)^{-1/2} + O(T^{o(1)})
\end{align*}
and thus
    $$ \left| \sum_{m} \psi(2\pi t_r/mN) (2\pi t_r/mN)^{1-\sigma_r} m^{-it_r} \right| \gg M^{1/2} N^{\sigma-1/2} T^{-\delta_2-o(1)},$$
    where $M \coloneqq 2\pi T/N \ll N^{1/2}$.  In particular
    $$ \sum_{m \in [M/10, 10 M]} \psi(2\pi t_r/mN) (2\pi t_r/mN)^{1-\sigma_r} m^{-it_r} \gg M^{\sigma} T^{-\delta_2-o(1)},$$
    since $N \gg T^{1/2}$ and $\sigma \geq 1/2$.  Performing a Fourier expansion as before, we conclude that
    $$ \sum_{m \in [M/10, 10 M]} m^{-it'_r} \ll M^{\sigma} T^{-\delta_2-o(1)}$$
    for some $t'_r = t_r + O(T^{o(1)})$, and one can argue as in the $N \ll T^{1/2}$ case (partitioning $[M/10, 10M]$ into $O(1)$ intervals each contained in some $[M',2M']$ with $M' \ll T^{1/2}$).

    Now suppose instead we are in the latter (``Type II'') case \eqref{td}.  We multiply both sides of \eqref{td} by the mollifier $\sum_{m \leq T^{\delta_2/2}} \frac{1}{m^{\sigma_r+it_r}}$ to obtain
    $$ \bigg| 1 + \sum_{T^{\delta_2/2} \leq n \leq T^{\delta_1+\delta_2/2}} \frac{a_n}{n^{\sigma_r+it_r}} \bigg| = o(1)$$
    where $a_n$ is some sequence with $a_n \ll T^{o(1)}$.  By dyadic decomposition and the pigeonhole principle, and refining the $t_r$ by a factor of $O(T^{o(1)})$ as needed, we can then find an interval $I$ in $[N,2N]$ with $T^{\delta_2/2} \ll N \ll T^{\delta_1+\delta_2/2}$ such that
    $$ \bigg| \sum_{n \in I} \frac{a_n}{n^{\sigma_r+it_r}} \bigg| \gg T^{-o(1)}$$
    and hence by Fourier expansion of $\frac{1}{n^{\sigma_r}}$ in $\log n$
    $$ \bigg| \sum_{n \in I} \frac{a_n}{n^{it'_r}} \bigg| \gg N^{\sigma_r} T^{-o(1)}$$
    for some $t'_r = t_r + O(T^{o(1)})$; by refining the $t_r$ by a further factor of $T^{o(1)}$ we may assume that the $t'_r$ are also $1$-separated; we can also pigeonhole so that $T = N^{\tau+o(1)}$ for some $\frac{1}{\delta_1+\delta_2/2} \leq \tau \leq \frac{1}{\delta_2/2}$.  Applying Lemma \ref{lv-def}, we conclude that
    $$ R \ll N^{\LV(\sigma,\tau)+o(1)} = T^{\LV(\sigma,\tau)/\tau+o(1)}$$
    and the claim follows in this case from \eqref{lv-bound}.
    \end{proof}

    Recently, a partial converse to the above lemma was established:

    \begin{lemma}[Large values from zero density]\label{zero-dens_implies_large}\cite[Theorem 1.2]{matomaki_teravainen_2024} If $\tau > 0$ and $1/2 \leq \sigma \leq 1$ are fixed, then
        $$ \LV_\zeta(\sigma,\tau)/\tau \leq \max\left( \frac{1}{2}, \sup_{\sigma \leq \sigma' \leq 1} \A(\sigma')(1-\sigma') + \frac{\sigma'-\sigma}{2} \right).$$
    \end{lemma}

    \begin{proof}  Let $N \geq 1$ be unbounded, $T = N^{\tau+o(1)}$, and $I \subset [N,2N]$ be an interval, and $t_1,\dots,t_R \in [T,2T]$ be $1$-separated with
    $$ \bigg| \sum_{n \in I} \frac{1}{n^{it_r}} \bigg| \gg N^{\sigma-o(1)}$$
    uniformly for all $r$.  By \cite[Theorem 1.2]{matomaki_teravainen_2024}, we have for any fixed $\delta>0$ that
    $$ R \ll T^\delta \sup_{\sigma-\delta \leq \sigma' \leq 1} T^{\frac{\sigma' - \sigma}{2}} N(\sigma', O(T)) + T^{\frac{1-\sigma}{2}+\delta}.$$
    Using Definition \ref{zero-def}, we conclude that
    $$ R \ll T^{\max( \frac{1}{2}, \sup_{\sigma-\delta \leq \sigma' \leq 1} \A(\sigma')(1-\sigma') + \frac{\sigma'-\sigma}{2} ) + O(\delta)}$$
    and thus
    $$ \LV_\zeta(\sigma,\tau) \leq \tau \max\left( \frac{1}{2}, \sup_{\sigma-\eps \leq \sigma' \leq 1} \A(\sigma')(1-\sigma') + \frac{\sigma'-\sigma}{2} \right) + O(\delta).$$
    Here, the implied constant  is understood to be uniform in $\delta$.
    Letting $\delta$ go to zero, and using the easily verified left-continuity of $\A$, we obtain the claim.
    \end{proof}

The suprema in Lemma \ref{zero-from-large} require unbounded values of $\tau$, but thanks to the ability to raise to a power, we can reduce to a bounded range of $\tau$.  Here is a basic such reduction, suited for machine-assisted proofs:

\begin{corollary}\label{zero-large-cor-0} Let $1/2 < \sigma < 1$ and $\tau_0 > 0$.  Then
    $$ \A(\sigma)(1-\sigma) \leq \max \left(\sup_{2 \leq \tau < \tau_0} \LV_\zeta(\sigma,\tau)/\tau, \sup_{\tau_0 \leq \tau \leq 2\tau_0} \LV(\sigma,\tau)/\tau\right)$$
    with the convention that the first supremum is $-\infty$ if it is vacuous (i.e., if $\tau_0 < 2$).
\end{corollary}

\begin{proof}
Denote the right-hand side by $B$, thus
    $$ \LV(\sigma,\tau) \leq B\tau$$
    for all $\tau_0 \leq \tau \leq 2\tau_0$, and
    \begin{equation}\label{lvz-b}
         \LV_\zeta(\sigma,\tau) \leq B\tau
    \end{equation}
    whenever $2 \leq \tau < \tau_0$.  From Lemma \ref{power-lemma} we then have
    $$ \LV(\sigma,\tau) \leq B\tau$$
    for all $k\tau_0 \leq \tau \leq 2k\tau_0$ and natural numbers $k$. Note that the intervals $[k\tau_0, 2k\tau_0]$ cover all of $[\tau_0,\infty)$, hence we have
    $$ \LV(\sigma,\tau) \leq B\tau$$
    for all $\tau \geq \tau_0$. In particular
    $$ \limsup_{\tau \to \infty} \LV(\sigma,\tau)/\tau  \leq B.$$
    Also, combining the previous estimate with \eqref{lvz-b} and \eqref{lvz-lv} we have
    \begin{equation*}\label{lvzo}
     \LV_\zeta(\sigma,\tau) \leq B\tau
    \end{equation*}
    for all $\tau \geq 2$.  By Lemma \ref{add-bound}(iii), this implies that
    $$ \LV_\zeta\bigg(\frac{1}{2} + \frac{1}{\tau-1} \left(\sigma-\frac{1}{2}\right), \frac{\tau}{\tau-1} \bigg) \leq B \frac{\tau}{\tau-1}$$
    for $\tau \geq 2$.  Thus
    $$ \sup_{\tau \geq 2} \frac{\LV_\zeta(\sigma,\tau)}{\tau} \leq B.$$
    The claim now follows from Lemma \ref{zero-from-large}.
\end{proof}

For machine assisted proofs, one can simply take $\tau_0$ to be a sufficiently large quantity, e.g., $\tau_0=3$ for $\sigma$ not too close to $1$, and larger for $\sigma$ approaching $1$, to recover the full power of Lemma \ref{zero-from-large}.  However, the amount of case analysis required increases with $\tau_0$.  For human written proofs, the following version of Corollary \ref{zero-large-cor-0} is more convenient:

\begin{corollary}\label{zero-large-cor} Let $1/2 < \sigma < 1$ and $\tau_0 > 0$.  Then
$$ \A(\sigma)(1-\sigma) \leq \max \left(\sup_{2 \leq \tau < 4\tau_0/3} \frac{\LV_\zeta(\sigma,\tau)}{\tau}, \sup_{2\tau_0/3 \leq \tau \leq \tau_0} \frac{\LV(\sigma,\tau)}{\tau}\right).$$
\end{corollary}

\begin{proof}  Applying Corollary \ref{zero-large-cor-0} with $\tau$ replaced by $4\tau_0/3$, it suffices to show that
$$\sup_{4\tau_0/3 \leq \tau \leq 8\tau_0/3} \frac{\LV(\sigma,\tau)}{\tau} \leq \sup_{2\tau_0/3 \leq \tau \leq \tau_0} \frac{\LV(\sigma,\tau)}{\tau}.$$
But this follows from Lemma \ref{power-lemma}, since the intervals $[2k\tau_0/3, k\tau_0]$ for $k=2,3$ cover all of $[4\tau_0/3,8\tau_0/3]$.
\end{proof}

The following special case of \Cref{zero-large-cor} is frequently (implicitly) used in practice to assist with the human readability of zero density proofs:

\begin{corollary}\label{zero-large-cor2} Let $1/2 < \sigma < 1$ and $\tau_0 > 0$.  Suppose that one has the bounds
\begin{equation}\label{lvo}
\LV(\sigma,\tau) \leq (3-3\sigma) \frac{\tau}{\tau_0}
\end{equation}
for $2\tau_0/3 \leq \tau \leq \tau_0$, and
\begin{equation}\label{lvoz}
 \LV_\zeta(\sigma,\tau) \leq (3-3\sigma) \frac{\tau}{\tau_0}
\end{equation}
for $2 \leq \tau < 4\tau_0/3$.  Then $\A(\sigma) \leq \frac{3}{\tau_0}$.
\end{corollary}

The reason why this particular special case is convenient is because the inequality
\begin{equation}\label{obvious}
    2 - 2\sigma \leq (3-3\sigma) \frac{\tau}{\tau_0}
\end{equation}
obviously holds for $\tau \geq 2\tau_0/3$.  That is to say, we automatically verify \eqref{lvo} in regimes where the Montgomery conjecture holds. In fact, we can do a bit better, thanks to subdivision:

\begin{corollary}\label{zero-large-cor3} Let $1/2 < \sigma < 1$ and $\tau_0 > 0$. Suppose that one has the bound \eqref{lvoz}
   for $2 \leq \tau < 4\tau_0/3$, and the Montgomery conjecture $\LV(\sigma,\tau) \leq 2-2\sigma$ whenever $0 \leq \tau \leq \tau_0+\sigma-1$.  Then $\A(\sigma) \leq \frac{3}{\tau_0}$.
\end{corollary}

\begin{proof}  We may assume that $\tau_0 \geq 3-3\sigma$, since otherwise we can use the bound
    $$ \A(\sigma)(1-\sigma) \leq \A(1/2)(1-1/2) = 1$$
coming from \eqref{adens} and the monotonicity properties of $N(\sigma,T)$.  By \Cref{hux-sub} we have
$$ \LV(\sigma,\tau) \leq \max( 2-2\sigma, 3-3\sigma + \tau - \tau_0 )$$
for all $\tau \geq 0$.  But both expressions on the right-hand side are bounded by $(3-3\sigma) \frac{\tau}{\tau_0}$ for $2\tau_3 \leq \tau \leq \tau_0$ and $\tau_0 \geq 3-3\sigma$, so the claim follows from Corollary \ref{zero-large-cor2}.
\end{proof}

With this machinery we can now quickly recover some existing zero density theorems, both conditional and unconditional. For instance, we have

\begin{theorem}\label{montgomery_implies_density} The Montgomery conjecture implies the density hypothesis.
\end{theorem}

\begin{proof} Apply Corollary \ref{zero-large-cor2} with $\tau_0=3/2$ (so that \eqref{lvoz} is vacuously true).
\end{proof}

\begin{theorem}\label{lindelof_implies_density} \cite{ingham_estimation_1940}, \cite{halasz_distribution_1969}
    The Lindel\"{o}f hypothesis implies the density hypothesis, and also that $\A(\sigma) \leq 0$ for $3/4 < \sigma \leq 1$.
\end{theorem}

\begin{proof} We will apply Corollary \ref{zero-large-cor}.  From \Cref{add-bound}(i) and the Lindel\"{o}f hypothesis we have
$$ \sup_{2 \leq \tau < 4\tau_0/3} \LV_\zeta(\sigma,\tau)/\tau\leq 0.$$
From Theorem \ref{l2-mvt} and Lemma \ref{power-lemma} we have
\begin{equation*}\label{ap}
     \LV(\sigma,\tau) \leq \max( (2-2\sigma)k, \tau + (1 - 2 \sigma)k )
\end{equation*}
for any natural number $k$ and $\tau \geq 1$; setting $k$ to be the integer part of $\tau$ we conclude in particular that
$$ \LV(\sigma,\tau) \leq (2-2\sigma)\tau + O(1),$$
and hence by taking $\tau_0$ large enough, we can make
$\sup_{2\tau_0/3 \leq \tau \leq \tau_0} \LV(\sigma,\tau)/\tau$ bounded by $2-2\sigma + \eps$ for any $\eps>0$.  This already gives the density hypothesis bound $\A(\sigma) \leq 2$.  For $\sigma > 3/4$, we may additionally apply \cite[Theorem 1]{halasz_distribution_1969} to obtain the Montgomery conjecture in this setting for all $\tau$, thus allowing one to make
$\sup_{2\tau_0/3 \leq \tau \leq \tau_0} \LV(\sigma,\tau)/\tau$ arbitrarily small, giving the bound $\A(\sigma) \leq 0$.
\end{proof}

\begin{theorem}[Ingham's bound]\label{thm:ingham_zero_density2}\cite{ingham_estimation_1940} For any $1/2 < \sigma < 1$, one has $\A(\sigma) \leq \frac{3}{2-\sigma}$.
\end{theorem}

\begin{proof}
We apply Corollary \ref{zero-large-cor3} with $\tau_0 := 2-\sigma$.  Here we have $4\tau_0/3 < 2$ since $\sigma>1/2$, so the claim \eqref{lvoz} is automatic; and the Montgomery conjecture hypothesis follows from Theorem \ref{l2-mvt}.
\end{proof}

\begin{theorem}[Huxley bound]\label{huxley-bound}\cite{Huxley} For any $1/2 < \sigma < 1$, one has $\A(\sigma) \leq \frac{3}{3\sigma-1}$. (In particular, the density hypothesis holds for $\sigma \geq 5/6$.)
\end{theorem}

\begin{proof} We apply Corollary \ref{zero-large-cor3} with $\tau_0 := 3\sigma-1$.  The Montgomery conjecture hypothesis follows from \eqref{huxley-lvt}. So it remains to show that \eqref{lvoz} holds for $2 \leq \tau < 4\tau_0/3$.  For $\sigma \leq 5/6$ we have $4\tau_0/3 \leq 2$, so the claim is vacuously true.  For $\sigma > 5/6$ we use \Cref{add-bound}(i), \Cref{exp-pair-mu} and the classical exponent pair $AB(0,1) = (\frac{1}{6},\frac{2}{3})$ to conclude that
$\LV_\zeta(\sigma,\tau) = -\infty$ whenever $\sigma > 1/2 + \tau/6$, but this is precisely $\tau < 6\sigma - 3$.  Since $6\sigma-3 > 4\tau_0/3$ when $\sigma > 5/6$, we obtain the claim.
\end{proof}

\begin{theorem}[Guth--Maynard bound]\label{guth-maynard-density} For any $1/2 < \sigma < 1$, one has $\A(\sigma) \leq \frac{15}{3+5\sigma}$.
\end{theorem}

\begin{proof} We may assume that $7/10 < \sigma < 8/10$, since the bound follows from the Ingham and Huxley bounds otherwise.  We apply Corollary \ref{zero-large-cor2} with $\tau_0 := \frac{3+5\sigma}{5}$.  We have $4\tau_0/3 < 2$, so the claim \eqref{lvoz} is vacuous and we only need to establish \eqref{lvo}.  We split into the subranges $13/5-2\sigma \leq \tau < \tau_0$ and $2\tau_0/3 \leq \tau \leq 13/5-2\sigma$. In the former range we use \eqref{guth-maynard-lvt} (and \eqref{obvious}), and reduce to showing that
$$ 18/5 - 4 \sigma \leq (3-3\sigma) \frac{\tau}{\tau_0},$$
and
$$ \tau + 12/5 - 4\sigma \leq (3-3\sigma) \frac{\tau}{\tau_0}$$
for $13/5-2\sigma \leq \tau < \tau_0$.  The first inequality follows from
\begin{equation}\label{18-5}
 18/5 - 4 \sigma \leq (3-3\sigma) \frac{13/5-2\sigma}{\tau_0}
\end{equation}
which can be verified numerically in the range $7/10 < \sigma < 8/10$.  Finally, the third inequality is actually an equality when $\tau=\tau_0$ and the right-hand side has a larger slope in $\tau$ than the left (since $\tau_0 \geq 3-3\sigma$), so the claim follows as well.

In the remaining region $2\tau_0/3 \leq \tau \leq 13/5-2\sigma$, we use Theorem \ref{l2-mvt} and \eqref{obvious} to reduce to showing that
$$ \tau + 1 - 2\sigma \leq (3-3\sigma) \frac{\tau}{\tau_0}$$
in this range.  This follows again from \eqref{18-5} which guarantees the inequality at the right endpoint $\tau = 13/5-2\sigma$.
\end{proof}

\subsection{New zero density estimates}

The previous zero density estimates were already known in the literature.  By computer-assisted optimisation, we were able to use the same set of tools to obtain some new zero density estimates.  The optimisation procedure is a higher-dimensional analogue of the program used to produce \Cref{beta-table}; however instead of considering feasible $(\alpha, \beta)$ tuples, one now combines known large value estimates into a $(\sigma, \tau, \LV(\sigma, \tau))$ feasible region (and zeta large value estimates into a $(\sigma, \tau, \LV_\zeta(\sigma,\tau))$ feasible region).  The quantities $\sup_{2\tau_0/3\le\tau\le\tau_0}\LV(\sigma, \tau)/\tau$ and $\sup_{2\le\tau\le4\tau_0/3}\LV_\zeta(\sigma,\tau)/\tau$ may then be computed as solutions to well-defined constrained optimisation problems. 

\begin{theorem}[Improved Heath-Brown zero density theorem]\label{hb-density2} For any $7/10 < \sigma \leq 1$, one has $$\A(\sigma) \leq \frac{3}{10\sigma-7}.$$
\end{theorem}

This result was previously obtained in the region $3/4 \leq \sigma \leq 25/28$ by Heath-Brown \cite{heathbrown_zero_1979}.  The main new ingredient is the exponent pair in \Cref{old-exp-pair} (leading in particular to \eqref{lvz-340}, which was not available
at the time of that paper).

\begin{proof}  We apply Corollary \ref{zero-large-cor3} with $\tau_0 := 10\sigma-7$.

To verify \eqref{lvo}, we now use Theorem \ref{hb-opt} and \eqref{obvious}, and reduce to showing that
$$ \tau + 10-13\sigma \leq (3-3\sigma) \frac{\tau}{\tau_0}$$
for $2\tau_0/3 \leq \tau \leq \tau_0$.  The inequality holds at $\tau=\tau_0$ since $\tau_0 \leq 10\sigma-7$, and hence for all smaller $\tau$ since $\tau_0 \geq 3-3\sigma$.

From \eqref{lvz-340} we have we have $\LV_\zeta(\sigma,\tau) = -\infty$ whenever $\sigma > \frac{3}{40} \tau + \frac{7}{10}$, or equivalently $\tau < \frac{4}{3} (10\sigma-7)$, which then immediately gives \eqref{lvoz}.
\end{proof}

Next, we recall a well-known result of Bourgain \cite{bourgain_large_2000} that the density hypothesis $\A(\sigma) \leq 2$ holds for all $\sigma > \frac{25}{32}$.  By abstracting and then using computer assistance to optimise the arguments in that paper, we were able to improve this bound.  Below is a translation of the computer-assisted argument to human-readable form.

\begin{theorem}[Improved Bourgain density hypothesis bound]\label{bourgain-density-improved} For fixed $17/22 \leq \sigma \leq 4/5$, one has 
$$\A(\sigma) \leq \max\left(\frac{2}{9\sigma-6}, \frac{9}{8(2\sigma-1)}\right).$$
Thus one has $\A(\sigma) \leq 2/(9\sigma-6)$ for $17/22 \leq \sigma \leq 38/49$ and $\A(\sigma) \leq 9/(8(2\sigma-1))$ for $38/49 \leq \sigma \leq 4/5$.
\end{theorem}

\begin{proof}
    We apply Corollary \ref{zero-large-cor3} with $\tau_0 := \min( \frac{ 9(3\sigma-2)}{2}, \frac{8(2\sigma-1)}{3} )$.  For future reference we note that
    \begin{equation}\label{slope}
        \frac{1}{3} < \frac{3-3\sigma}{\tau_0} < \frac{1}{2}.
    \end{equation}

    It suffices to show that
    \begin{equation}\label{lv-st}
     \LV(\sigma,\tau) \leq \frac{\tau}{\tau_0} (3-3\sigma)
    \end{equation}
    for $2\tau_0/3 \leq \tau \leq \tau_0$, as well as
    \begin{equation}\label{lvz-st}
    \LV_\zeta(\sigma,\tau) \leq \frac{\tau}{\tau_0} (3-3\sigma)
    \end{equation}
    for $2 \leq \tau < 4\tau_0/3$.
    For \eqref{lvz-st} we use the twelfth moment bound \eqref{twelfth-bound}.  Since the slope of
    $2\tau - 12 (\sigma-1/2)$ in $\tau$ exceeds that of $\frac{\tau}{\tau_0} (3-3\sigma)$ by \eqref{slope}, it suffices to check the bound at the endpoint, i.e., to show that
    $$ 8\tau_0/3 - 12 (\sigma-1/2) \leq 4 -4\sigma$$
    or equivalently $\tau_0 \leq \frac{3(4\sigma-1)}{4}$, which one can easily check to be the case.

    Now we prove \eqref{lv-st}.  Set $\rho := \LV(\sigma,\tau)$.  From the $k=3$ case of \eqref{jutila-lvt} we have
    \begin{equation}\label{rhomax}
        \rho \leq \max \bigg(2-2\sigma, \tau + \frac{10-16\sigma}{3}, \tau + 18-24\sigma \bigg)
    \end{equation}
    which implies the required bound $\rho \leq \frac{\tau}{\tau_0}(3-3\sigma)$ unless one has
    \begin{equation}\label{tau-lower}
        \tau \geq - \max\left( \frac{10-16\sigma}{3}, 18-24\sigma\right) / \left(1 - \frac{3-3\sigma}{\tau_0}\right).
    \end{equation}
    In this regime, one can also check from \eqref{rhomax} that
    $$ \rho \leq \min(1, 4-2\tau)$$
    (with room to spare) so we may apply \Cref{bourgain-lvt}(ii) to obtain
    \begin{equation*}\label{40-alt}
    \begin{split}
    \rho &\leq \max( \alpha_2 + 2 - 2 \sigma, \alpha_1+\alpha_2/2 + 2-2\sigma, -\alpha_2 + 2\tau+4-8\sigma, \\
    &\qquad\qquad -2\alpha_1 + \tau + 12 - 16 \sigma, 4\alpha_1 + 2+\max(1,2\tau-2)-4\sigma)
    \end{split}
    \end{equation*}
    for any $\alpha_1, \alpha_2 \geq 0$.

    We first consider the case when $38/49 \leq \sigma \leq 4/5$, so that $\tau_0 = 8(2\sigma-1)/3$.
    We set
    $$ \alpha_2 := \max\left( \frac{5\tau}{4} - (1+\sigma), 0\right)$$
    and
    $$ \alpha_1 := \frac{\tau}{3} - \frac{2}{3} (7\sigma-5) - \alpha_2/6.$$
    With this choice, the expressions $\alpha_1+\alpha_2/2 + 2-2\sigma$ and $-2\alpha_1 + \tau + 12 - 16 \sigma$ are both equal to
    $\frac{\tau}{3} + \frac{16-20\sigma}{3} + \frac{\alpha_2}{3}$, while $-\alpha_2 + 2\tau+4-8\sigma$ is equal to
    $$ \frac{\tau}{3} + \frac{16-20\sigma}{3} + \frac{\alpha_2}{3} - \frac{4}{3} \left(\alpha_2 - \frac{5\tau}{4} - (1+\sigma)\right)$$
    which is less than or equal to the previous expression by definition of $\alpha_2$.  Finally, the expression
    $4\alpha_1 + 1+\max(2,2\tau-2)-4\sigma$ is equal to
    $$ \frac{4\tau}{3} + \frac{46-68\sigma}{3} + \max(1,2\tau-2) - \frac{2\alpha_2}{3}.$$
    Thus it remains to show the bounds
    \begin{equation}\label{bound-1}
     \alpha_2 + 2 - 2 \sigma \leq \frac{\tau}{\tau_0} (3-3\sigma)
    \end{equation}
    \begin{equation}\label{taut}
        \frac{\tau}{3} + \frac{16-20\sigma}{3} + \frac{\alpha_2}{3} \leq \frac{\tau}{\tau_0} (3-3\sigma)
    \end{equation}
    and
    \begin{equation}\label{tar}
        \frac{4\tau}{3} + \frac{46-68\sigma}{3}+ \max(1,2\tau-2) - \frac{2\alpha_2}{3} \leq \frac{\tau}{\tau_0} (3-3\sigma).
    \end{equation}
    For \eqref{bound-1}, we trivially have $2-2\sigma \leq \frac{\tau}{\tau_0}(3-3\sigma)$ since $\tau \geq 2\tau_0/3$, and the slope of $\frac{5\tau}{4} - (1+\sigma) + 2-2\sigma$ in $\tau$ certainly exceeds $\frac{3-3\sigma}{\tau_0}$ by \eqref{slope}, so it suffices to check the endpoint
    $$ \frac{5\tau_0}{4} - (1+\sigma) + 2-2\sigma \leq 3-3\sigma$$
    which one can check to be valid for $\sigma \leq 4/5$. Now we turn to \eqref{taut}.  From \eqref{slope} it suffices to show that
    $$ \frac{\tau_0}{3} + \frac{16-20\sigma}{3} + \frac{\frac{5\tau_0}{4} - (1+\sigma)}{3} \leq 3 -3 \sigma$$
    and
    $$ \frac{\tau}{3} + \frac{16-20\sigma}{3} + \frac{0}{3} \leq 2 -2 \sigma.$$
    The former is an identity, and the latter simplifies to $\tau \geq 14\sigma-10$, which one can check follows from \eqref{tau-lower} (with some room to spare) in the regime $38/49 \leq \sigma \leq 4/5$, giving the claim.  Finally, for \eqref{tar} it suffices to show that
    $$ \frac{4\tau}{3} + \frac{46-68\sigma}{3}  + \max(1,2\tau_0-2) - \frac{2(\frac{5\tau}{4} - (1+\sigma))}{3} \leq \frac{\tau}{\tau_0} (3-3\sigma)$$
    which by \eqref{slope} would follow from
    $$ \frac{4\tau_0}{3} + \frac{46-68\sigma}{3} + \max(1,2\tau_0-2) - \frac{2(\frac{5\tau_0}{4} - (1+\sigma))}{3} \leq 3-3\sigma$$
    and one can check that this applies for $\sigma \geq 38/49$.

    Now suppose that we are in the case $16/21 \leq \sigma \leq 38/49$, so that $\tau_0 = \frac{9(3\sigma-2)}{2} \leq \frac{72}{49} < \frac{3}{2}$ (so in particular $\max(1,2\tau-2) = 1$ for $\tau \leq \tau_0$).  We set
    $$ \alpha_2 := \max(11 - 16 \sigma + \tau,0)$$
    and
    $$ \alpha_1 := \frac{\tau}{3} - \frac{2}{3} (7\sigma-5) - \alpha_2/6.$$
    Note that for $\sigma \leq 38/49$, one has
    $$ (5\tau/4 - (1+\sigma)) - (11 - 16 \sigma + \tau) \leq 15\sigma - 12 - \tau_0/4 \leq 0$$
    and hence
    $$ \alpha_2 \geq 5\tau/4 - (1+\sigma).$$
    As before, we conclude that the quantities $\alpha_1+\alpha_2/2 + 2-2\sigma$ and $-2\alpha_1 + \tau + 12 - 16 \sigma$ are both equal to $\frac{\tau}{3} + \frac{16-20\sigma}{3} + \frac{\alpha_2}{3}$, while $-\alpha_2 + 2\tau+4-8\sigma$ does not exceed this quantity.  Thus it suffices to show \eqref{bound-1}, \eqref{taut}, \eqref{tar} as before.

    For \eqref{bound-1} we argue as before to reduce to showing that
    $$ 11 - 16 \sigma + \tau_0 + 2 - 2 \sigma \leq 3 - 3\sigma$$
    which one can check to be true (with room to spare) for $\sigma \geq 16/21$.  For \eqref{taut}, we use \eqref{slope} as before to reduce to showing that
    $$ \frac{\tau_0}{3} + \frac{16-20\sigma}{3} + \frac{11 - 16 \sigma + \tau}{3} \leq 3 -3 \sigma$$
    and
    $$ \frac{\tau}{3} + \frac{16-20\sigma}{3} + \frac{0}{3} \leq 2 -2 \sigma.$$
    The first inequality is an identity, and the latter again reduces to  $\tau \geq 14\sigma-10$ which one can check follows from \eqref{tau-lower}.  For \eqref{tar} it suffices to show that
    $$ \frac{4\tau}{3} + \frac{46-68\sigma}{3}  + 1 - \frac{2(11 - 16 \sigma + \tau)}{3} \leq \frac{\tau}{\tau_0} (3-3\sigma)$$
    which by \eqref{slope} follows from
    $$ \frac{4\tau_0}{3} + \frac{46-68\sigma}{3}  + 1 - \frac{2(11 - 16 \sigma + \tau_0)}{3} \leq 3 - 3\sigma,$$
    but this is an identity.
\end{proof}

In \cite{bourgain_remarks_1995} the following zero density theorem was proven:

\begin{theorem}[Bourgain zero density theorem]\label{bourgain-zd}\cite[Proposition 3]{bourgain_remarks_1995}  Let $(k,\ell)$ be an exponent pair with $k < 1/5$, $\ell > 3/5$, and $15\ell + 20k > 13$.  Then, for any $\sigma > \frac{\ell+1}{2(k+1)}$, one has
    $$ \A(\sigma) \leq \frac{4k}{2(1+k)\sigma - 1 - \ell}$$
    assuming either that $k < \frac{11}{85}$, or that $\frac{11}{85} < k < \frac{1}{5}$ and $\sigma > \frac{144k-11\ell-11}{170k -22}$.
\end{theorem}

This result was combined with some known exponent pairs to obtain new zero density bounds on $\A(\sigma)$.
It was remarked in that paper that further exponent pairs could be used to improve the bounds.  This is easy to achieve computationally in our framework, and we record the results here:

\begin{theorem}[Optimised Bourgain zero density bound]\label{bourgain-zero-density-optimized} For fixed $\sigma$, one has
    \[
    \A(\sigma) \leq \begin{cases}
\dfrac{11}{12(4 \sigma - 3)} & \dfrac{3}{4} < \sigma \le \dfrac{14}{15},\\
\dfrac{391}{2493 \sigma - 2014} & \dfrac{14}{15} < \sigma \le \dfrac{2841}{3016},\\
\dfrac{22232}{163248 \sigma - 134765} & \dfrac{2841}{3016} < \sigma \le \dfrac{859}{908},\\
\dfrac{356}{2742 \sigma - 2279} & \dfrac{859}{908} < \sigma \le \dfrac{1625}{1692},\\
\dfrac{2609588}{20732766\sigma - 17313767} & \dfrac{1625}{1692} < \sigma \le \dfrac{3334585}{3447984},\\
\dfrac{75872}{9 (81024 \sigma - 69517)} & \dfrac{3334585}{3447984} < \sigma \le \dfrac{974605}{1005296},\\
\dfrac{288}{3616 \sigma - 3197} & \dfrac{974605}{1005296} < \sigma \le \dfrac{5857}{6032},\\
\dfrac{86152}{1447460 \sigma - 1311509} & \dfrac{5857}{6032} < \sigma < 1.
\end{cases}
    \]
\end{theorem}

\begin{proof}
Let $\mathcal{S}(\sigma)$ denote the closure of the region
    \begin{align*}
    \Bigg\{(k, \ell) : 0 < k < \frac{1}{5}, \frac{3}{5} < \ell < 1, 15\ell + 20k > 13, \frac{\ell + 1}{2(k + 1)} < \sigma, \\
    k < \frac{11}{85} \text{ or } \left(k > \frac{11}{85} \text{ and } \frac{144k - 11\ell - 11}{170k - 22} < \sigma\right)\Bigg\}.
    \end{align*}
    One may verify that $\mathcal{S}(\sigma)$ is a convex polygon for all $3/4 < \sigma < 1$, and thus so is $\mathcal{S}(\sigma) \cap H$, where $H$ is the convex hull of exponent pairs. Thus
    \[
    \min_{(k, \ell) \in \mathcal{S}(\sigma) \cap H}\frac{4k}{2(1 + k)\sigma - 1 - \ell}
    \]
    is a convex optimisation problem for each $3/4 < \sigma < 1$.
    We take the following choices of $(k, \ell)$ (found with the aid of computer assistance).
    \[
    \left(\frac{11}{85}, \frac{59}{85}\right),
    \left(\frac{391}{4595}, \frac{3461}{4595}\right),
    \left(\frac{2779}{38033}, \frac{58699}{76066}\right),
    \left(\frac{89}{1282}, \frac{997}{1282}\right),
    \]
    \[
    \left(\frac{652397}{9713986}, \frac{7599781}{9713986}\right),
    \left(\frac{2371}{43205}, \frac{280013}{345640}\right),
    \left(\frac{9}{217}, \frac{1461}{1736}\right),
    \left(\frac{10769}{351096}, \frac{609317}{702192}\right).
    \]
    The desired result follows from applying \Cref{bourgain-zd} and taking a minimum over the implied bounds. Note that sharper bounds are possible close to $\sigma = 1$ by choosing other exponent pairs, however it turns out such results are superseded by other zero density estimates. 
\end{proof}

We record the current best known zero density estimates proven in the literature, or in this paper in \Cref{zero_density_estimates_table}.

\begin{table}[ht]
    \resizebox{\textwidth}{!}{
    \def\arraystretch{1}
    \centering
    \begin{tabular}{|c|c|c|}
    \hline
    $\A(\sigma)$ bound & Range & Reference\\
    \hline
    $\dfrac{3}{2 - \sigma}$ & $\dfrac{1}{2} \leq \sigma \le \dfrac{7}{10} = 0.7$ & Theorem \ref{thm:ingham_zero_density2}*\\
    \hline
    $\dfrac{15}{3+5\sigma}$ & $\dfrac{7}{10} \leq \sigma < \dfrac{19}{25} = 0.76$ & Theorem \ref{guth-maynard-density}\\
    \hline
    $\dfrac{9}{8\sigma - 2}$ & $\dfrac{19}{25} \leq \sigma < \dfrac{127}{167} = 0.7604\ldots$ & \cite{ivic_exponent_pairs}, \cite[Theorem 11.4]{ivic}\\
    \hline
    $\dfrac{15}{13\sigma - 3}$ & $\dfrac{127}{167} \leq \sigma < \dfrac{13}{17} = 0.7647\ldots$ & \cite{ivic_exponent_pairs}, \cite[Theorem 11.4]{ivic}\\
    \hline
    $\dfrac{6}{5\sigma - 1}$ & $\dfrac{13}{17} \leq \sigma < \dfrac{17}{22} = 0.7727\ldots$ & \cite{ivic_exponent_pairs}, \cite[Theorem 11.4]{ivic}*\\
    \hline
    $\dfrac{2}{9\sigma - 6}$ & $\dfrac{17}{22} \leq \sigma < \dfrac{41}{53} = 0.7735\ldots$ & \textbf{Theorem \ref{bourgain-density-improved}}\\
    \hline
    $\dfrac{9}{7\sigma - 1}$ & $\dfrac{41}{53} \leq \sigma < \dfrac{7}{9} = 0.7777\ldots$ & \cite{ivic_exponent_pairs}, \cite[Theorem 11.4]{ivic}\\
    \hline
    $\dfrac{9}{8(2\sigma-1)}$ & $\dfrac{7}{9} \le \sigma < \dfrac{1867}{2347} = 0.7954\ldots$ & \textbf{Theorem \ref{bourgain-density-improved}} \\
    \hline
    $\dfrac{3}{2\sigma}$ & $\dfrac{1867}{2347} \le \sigma < \dfrac{4}{5} = 0.8$ & \cite{bourgain_dirichlet_2000}* \\
    \hline
    $\dfrac{3}{2\sigma}$ & $\dfrac{4}{5} \le \sigma < \dfrac{7}{8} = 0.875$ & \cite{ivic_exponent_pairs}, \cite[Theorem 11.4]{ivic}* \\
    \hline
    $\dfrac{3}{10\sigma - 7}$ & $\dfrac{7}{8} \le \sigma < \dfrac{279}{314} = 0.8885\ldots$ & \cite{heathbrown_zero_1979}*\\
    \hline
    $\dfrac{24}{30\sigma - 11}$ & $\dfrac{279}{314} \le \sigma < \dfrac{155}{174} = 0.8908\ldots$ & \cite{chen_debruyne_vindas_density_2024}*  \\
    \hline
    $\dfrac{24}{30\sigma - 11}$& $\dfrac{155}{174} \le \sigma \le \dfrac{9}{10} = 0.9$ & \cite{ivic_exponent_pairs}, \cite[Theorem 11.2]{ivic}*\\
    \hline
    $\dfrac{3}{10\sigma - 7}$ & $\dfrac{9}{10} < \sigma \le \dfrac{31}{34} = 0.9117\ldots$ & \textbf{Theorem \ref{hb-density2}}\\
    \hline
    $\dfrac{11}{48\sigma - 36}$ & $\dfrac{31}{34} < \sigma < \dfrac{14}{15} = 0.9333\ldots$ & \textbf{Corollary \ref{bourgain-zero-density-optimized}} \\
    \hline
    $\dfrac{391}{2493\sigma - 2014}$ & $\dfrac{14}{15} \le \sigma < \dfrac{2841}{3016} = 0.9419\ldots$ & \textbf{Corollary \ref{bourgain-zero-density-optimized}} \\
    \hline
    $\dfrac{22232}{163248\sigma - 134765}$ & $\dfrac{2841}{3016} \le \sigma < \dfrac{859}{908} = 0.9460\ldots$ & \textbf{Corollary \ref{bourgain-zero-density-optimized}} \\
    \hline
     $\dfrac{356}{2742\sigma - 2279}$ & $\dfrac{859}{908} \le \sigma < \dfrac{23}{24} = 0.9583\ldots$ & \textbf{Corollary \ref{bourgain-zero-density-optimized}} \\
    \hline
    $\dfrac{3}{24\sigma - 20}$ & $\dfrac{23}{24} \leq \sigma < \dfrac{2211487}{2274732} = 0.9721\ldots$ & \cite[Theorem 1]{pintz_density_2023}* \\
    \hline
    $\dfrac{86152}{1447460\sigma - 1311509}$ & $\dfrac{2211487}{2274732} \le \sigma < \dfrac{39}{40} = 0.975$ & \textbf{Corollary \ref{bourgain-zero-density-optimized}} \\
    \hline
    $\dfrac{2}{15 \sigma - 12}$ & $\dfrac{39}{40} \leq \sigma < \dfrac{41}{42} = 0.9761\ldots$ & \cite[Theorem 1]{pintz_density_2023}* \\
    \hline
    $\dfrac{3}{40 \sigma - 35}$ & $\dfrac{41}{42} \leq \sigma < \dfrac{59}{60} = 0.9833\ldots$ & \cite[Theorem 1]{pintz_density_2023}* \\
    \hline
    $\dfrac{3}{n(1 - 2(n - 1)(1 - \sigma))}$ & \begin{tabular}{@{}c@{}}$1 - \dfrac{1}{2n(n - 1)} \le \sigma < 1 - \dfrac{1}{2n(n + 1)}$\\(for integer $n \ge 6$)\end{tabular} & \cite[Theorem 1]{pintz_density_2023}* \\
    \hline
    \end{tabular}
    }
    \caption{Current best upper bound on $\A(\sigma)$.  The new results in this paper are marked in bold. Results labelled with an asterisk also appear in \cite[Table 2]{trudgian-yang}.}
    \label{zero_density_estimates_table}
    \end{table}

\begin{figure}
\centering
\includegraphics[width=10cm]{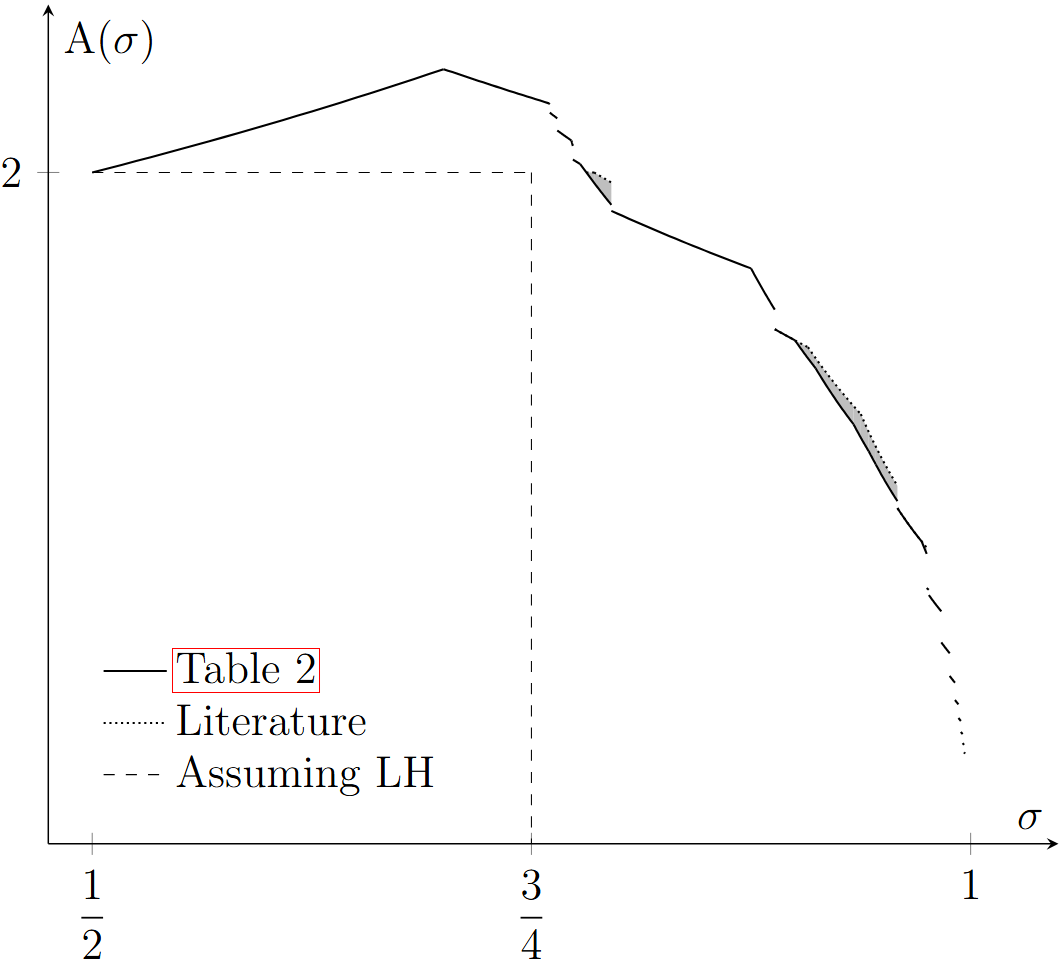}
\caption{The bounds in \Cref{zero_density_estimates_table}, compared against the existing literature bounds on $\A(\sigma)$, with differences shaded grey.}
\label{fig:zero_density_estimate}
\end{figure}

\subsection{A heuristic for zero density estimates}\label{heuristic-sec}

We can now state a rough heuristic as to what zero density estimates to expect from a given large value theorem:

\begin{heuristic}[Predicting a zero density estimate from a large value theorem]\label{lv-heuristic} Suppose that $1/2 \leq \sigma \leq 1$ and $\tau_0 \geq 1$ are such that one can prove $\LV(\sigma, \tau_0) \leq 3-3\sigma$ (i.e., the Montgomery conjecture holds here with a multiplicative loss of $3/2$).  Then in principle, one can hope to prove $\A(\sigma) \leq 3/\tau_0$.  Conversely, if one cannot prove $\LV(\sigma, \tau_0) \leq 3-3\sigma$, then the bound $\A(\sigma) \leq 3/\tau_0$ is likely out of reach.
\end{heuristic}

We justify this heuristic as follows, though we stress that the arguments that follow are not fully rigorous.  In the first part, we simply apply Corollary \ref{zero-large-cor2}.  In practice, the \eqref{lvo} is often more delicate than \eqref{lvoz} and ends up being the limiting factor for the bounds; furthermore, within \eqref{lvo}, it is the right endpoint $\tau=\tau_0$ of the range $2\tau_0/3 \leq \tau \leq \tau_0$ that ends up being the bottleneck; but this is precisely the claimed criterion $\LV(\sigma, \tau_0) \leq 3-3\sigma$.  We remark that in some cases (particularly for $\sigma$ close to one), the estimate \eqref{lvoz} ends up being more of the bottleneck than \eqref{lvo}, and so one should view $3/\tau_0$ here as a theoretical upper limit of methods rather than as a guaranteed bound.  (In particular, the need to  establish the bound $\LV_\zeta(\sigma, \frac{4}{3}\tau_0-\eps) < 4-4\sigma$ for $\eps>0$ small can sometimes be a more limiting factor.)

Conversely, suppose that
\begin{equation}\label{lst3}
    \LV(\sigma,\tau_0) > 3-3\sigma,
\end{equation}
but that one still wants to prove the bound $\A(\sigma) \leq 3/\tau_0$. Heuristically, Theorem \ref{zero-dens_implies_large} suggests that in order to do this, it is necessary to establish the bound $\LV_\zeta(\sigma,\tau)/\tau \leq \frac{3}{\tau_0}(1-\sigma)$ for all $\tau \geq 2$.  In particular, one should show that
$$ \LV_\zeta(\sigma,2\tau_0) \leq 6-6\sigma.$$
Let us consider the various options one has to do this.  There are ways to control zeta large values that do not apply to general large value estimates, such as moment estimates of the zeta-function, exponent pairs, or control of $\beta$ and $\mu$.  However, at our current level of understanding, these techniques only control $\LV_\zeta(\sigma,\tau)$ for relatively small values of $\tau$, and in practice $2\tau_0$ is too large for these methods to apply; this exponent also tends to be too large for direct application of standard large value theorems to be useful.  Hence, the most viable option in practice is raising to a power (Lemma \ref{power-lemma}), using
$$ \LV_\zeta(\sigma,2\tau_0) \leq k \LV_\zeta(\sigma,2\tau_0/k)$$
for some natural number $k \geq 2$.  However, the most natural choice $k=2$ is blocked due to our hypothesis \eqref{lst3}, while in practice the $k \geq 3$ choice is blocked because of Lemma \ref{lv-lower}.  Hence it appears heuristically quite difficult to establish $\A(\sigma) \leq 3/\tau_0$ with current technology, in the event that \eqref{lst3} occurs.

In \Cref{zero_density_heuristic} we show how some of the large value theorems proven above are in agreement with \Cref{lv-heuristic} for at least some ranges of $\sigma$.  Further examples can be found in the ANTEDB.

\begin{table}[ht]
    \resizebox{\textwidth}{!}{
    \def\arraystretch{1.3}
    \centering
    \begin{tabular}{|c|c|c|c|}
    \hline
    Large value theorem & Predicted $\tau_0$ & Predicted $\A(\sigma)$ & Zero density theorem\\
    \hline
    Theorem \ref{l2-mvt} & $2-\sigma$ & $\frac{3}{2-\sigma}$ & Theorem \ref{thm:ingham_zero_density2}\\
    \hline
    \eqref{huxley-lvt} & $3\sigma-1$ & $\frac{3}{3\sigma-1}$ & Theorem \ref{huxley-bound} \\
    \hline
    \eqref{hb-opt} & $10\sigma-7$ & $\frac{3}{10\sigma-7}$ & Theorems \ref{hb-density2} \\
    \hline
    \eqref{guth-maynard-lvt} & $\frac{5\sigma-3}{3}$ & $\frac{15}{5\sigma-3}$ & Theorem \ref{guth-maynard-density} \\
    \hline
    \end{tabular}}
    \caption{Examples of large value theorems, the values of $\tau_0$ and $\A(\sigma)$ they suggest, and rigorous zero density theorems that attain the predicted value for at least some ranges of $\sigma$.}
    \label{zero_density_heuristic}
\end{table}

\section{Additive energy}

Using the concept of additive energy from \Cref{energy-def}, we can define additive energy versions of the exponents $\LV(\sigma,\tau)$, $\LV_\zeta(\sigma,\tau)$, and $\A(\sigma)$:

\begin{definition}[Large value exponent for energy]\label{lve-def} Let $1/2 \leq \sigma \leq 1$ and $\tau \geq 0$ be fixed. We define $\LV^*(\sigma,\tau)$ to be the least fixed quantity for which the following claim is true: whenever $(N,T,V,(a_n)_{n \in [N,2N]},J,W)$ is a large value pattern with $N>1$ unbounded, $T = N^{\tau+o(1)}$, and $V = N^{\sigma+o(1)}$, then
    $$ E_1(W) \ll N^{\LV^*(\sigma,\tau)+o(1)}.$$
    We define $\LV^*_\zeta(\sigma,\tau)$ similarly, but where $(N,T,V,(a_n)_{n \in [N,2N]},J,W)$ is now required to be a zeta large value pattern.
\end{definition}

\begin{definition}[Zero density exponents for energy]\label{zeroe-def}  For $\sigma \in \R$ and $T>0$, let $N^*(\sigma,T)$ denote the additive energy of the imaginary parts of zeroes $\rho$ of the Riemann zeta-function with $\mathrm{Re}(\rho) \geq \sigma$ and $|\mathrm{Im}(\rho)| \leq T$.

If $1/2 \leq \sigma < 1$ is fixed, we define the zero density exponent $\A^*(\sigma) \in [-\infty,\infty)$ to be the infimum of all (fixed) exponents $A^*$ for which one has
        $$ N^*(\sigma-\delta,T) \ll T^{A^* (1-\sigma)+o(1)}$$
    whenever $T$ is unbounded and $\delta>0$ is infinitesimal.
\end{definition}

\begin{remark}  It is theoretically possible (though not believed to be likely) that two zeroes of $\zeta$ to the right of the critical line have the same imaginary part, in which case the set of imaginary parts may contain multiplicity.  However, by the Riemann--von Mangoldt formula, this multiplicity can only affect the additive energy by a factor of $O(\log^{O(1)} T)$ at most, and thus has no impact on the exponent $\A^*(\sigma)$.
\end{remark}

From the easy bounds
$$ |W|^2 \ll E_1(W) \ll |W|^3$$
we see that
\begin{equation}\label{lve-triv}
    2\LV(\sigma,\tau) \leq \LV^*(\sigma,\tau) \leq 3\LV(\sigma,\tau)
\end{equation}
and
$$ 2\LV_\zeta(\sigma,\tau) \leq \LV^*_\zeta(\sigma,\tau) \leq 3\LV_\zeta(\sigma,\tau)$$
and also
$$ 2\A(\sigma) \leq \A^*(\sigma) \leq 3\A(\sigma).$$
We also trivially have $\LV^*_\zeta(\sigma,\tau) \leq \LV^*(\sigma,\tau)$.

In order to record the interactions between $\LV^*(\sigma,\tau)$, $\LV^*_\zeta(\sigma)$ and large value theorems, we introduce the following regions.

\begin{definition}[Large value energy region]\label{lv-edef} The \emph{large value energy region} $\Energy \subset \R^5$ is defined to be the set of all fixed tuples $(\sigma,\tau,\rho,\rho^*,s)$ with $1/2 \leq \sigma \leq 1$, $\tau, \rho, \rho' \geq 0$, such that there exists a large value pattern $(N,T,V,(a_n)_{n \in [N,2N]},J,W)$ with $N>1$ unbounded, $V = N^{\sigma+o(1)}$, $T = N^{\tau+o(1)}$, $V = N^{\sigma+o(1)}$, $|W| = N^{\rho+o(1)}$, $E_1(W) = N^{\rho^*+o(1)}$ and $S(N,W) = N^{s+o(1)}$.

    We define the \emph{zeta large value energy region} $\Energy_\zeta \subset \R^5$ similarly, but where now $(N,T,V,(a_n)_{n \in [N,2N]},J,W)$ is required to be a zeta large value pattern.
\end{definition}

By the usual arguments, one can define this region without using asymptotic notation:

\begin{lemma}[Non-asymptotic form of large value energy region]\label{lve-asymp} Let $1/2 \leq \sigma \leq 1$, $\tau \geq 0$,  $\rho, \rho^* \geq 0$, and $s \in \R$ be fixed.  Then the following are equivalent:
    \begin{itemize}
    \item[(i)] $(\sigma,\tau,\rho,\rho^*,s) \in \Energy$.
    \item[(ii)] For every $\eps>0$ and $C > 0$, there exists a large value pattern $$(N,T,V,(a_n)_{n \in [N,2N]},J,W)$$ with $N \geq C$, $N^{\tau-\delta} \leq T \leq N^{\tau+\delta}$, $N^{\sigma-\delta} \leq V \leq N^{\sigma+\delta}$,
    $N^{\rho-\eps} \leq |W| \leq N^{\rho+\eps}$,
    $N^{\rho^*-\eps} \leq E_1(W) \leq N^{\rho^*+\eps}$, and
    $N^{s-\eps} \leq S(N, W) \leq N^{s+\eps}$.
    \end{itemize}
    A similar statement is true with $\Energy$ replaced by $\Energy_\zeta$, and with $(N,T,V,(a_n)_{n \in [N,2N]},J,W)$ required to be a zeta large value pattern.
\end{lemma}

The quantities $\LV^*(\sigma,\tau)$, $\LV^*_\zeta(\sigma,\tau)$ can now be expressed in terms of these regions by the formulae
\begin{equation}\label{lvs-energy}
     \LV^*(\sigma,\tau) = \sup \{ \rho^*: (\sigma,\tau,\rho,\rho^*,s) \in \Energy\}
\end{equation}
and
\begin{equation}\label{lvsz-energy}
     \LV^*_\zeta(\sigma,\tau) = \sup \{ \rho^*: (\sigma,\tau,\rho,\rho^*,s) \in \Energy_\zeta\}.
\end{equation}

We have the following analogue of \Cref{zero-from-large}:

\begin{lemma}[Zero density energy from large values energy]\label{zeroe-from-large}  Let $1/2 < \sigma < 1$.  Then
    $$ \A^*(\sigma)(1-\sigma) \leq \max\left( \sup_{\tau \geq 1} \LV^*_\zeta(\sigma,\tau)/\tau, \limsup_{\tau \to \infty} \LV^*(\sigma,\tau)/\tau \right).$$
    \end{lemma}

    \begin{proof}
    Write the right-hand side as $B$, then $B \geq 0$ (from \eqref{lve-triv}) and we have
    \begin{equation}\label{lvze-bound}
        \LV^*_\zeta(\sigma,\tau) \leq B \tau
    \end{equation}
    for all $\tau \geq 1$, and
    \begin{equation}\label{lve-bound}
        \LV^*(\sigma,\tau) \leq (B+\eps) \tau
    \end{equation}
    whenever $\eps>0$ and $\tau$ is sufficiently large depending on $\eps$ (and $\sigma$).  It would suffice to show, for any $\eps>0$, that $N^*(\sigma,T) \ll T^{B+O(\eps)+o(1)}$ for unbounded $T$.

    By dyadic decomposition, it suffices to show for unbounded $T$ that the additive energy of imaginary parts of zeroes in $[T,2T]$ is $\ll T^{B+O(\eps)+o(1)}$.  As in the proof of Lemma \ref{zero-from-large}, we can assume the imaginary parts are $1$-separated (here we take advantage of the triangle inequality \eqref{energy-triangle} for additive energy).

    Suppose there is a zero $\sigma'+i t$ of this form.  Then by  \cite[Theorem 1.8]{ivic}, one has
    $$ \sum_{n \leq T} \frac{1}{n^{\sigma'+it}} \ll T^{-1}.$$
    Let $0 < \delta_1 < \eps$ be a small fixed quantity to be chosen later, and let $0 < \delta_2 < \delta_1$ be sufficiently small depending on $\delta_1,\delta_2$.  By the triangle inequality, and refining the sequence $t'$ by a factor of at most $2$, we either have
    $$ \bigg|\sum_{T^{\delta_1} \leq n \leq T} \frac{1}{n^{\sigma'+it}} \bigg| \gg T^{-\delta_2}$$
    for all zeroes, or \eqref{td} for all zeroes.

    Suppose we are in the former (``Type I'') case, we can dyadically partition and conclude from the pigeonhole principle that
    $$ \bigg| \sum_{n \in I} \frac{1}{n^{\sigma'+it}} \bigg| \gg T^{-\delta_2-o(1)}$$
    for some interval $I$ in some $[N,2N]$ with $T^{\delta_1} \ll N \ll T$, with at most $O(\log T)$ different choices for $I$.  Performing a Fourier expansion of $n^{\sigma'}$ in $\log n$ and using the triangle inequality one can then deduce that
    $$ \bigg| \sum_{n \in I} \frac{1}{n^{it'}} \bigg| \gg N^{\sigma'} T^{-\delta_2-o(1)}$$
    for some $t' = t + O(T^{o(1)})$; refining the $t$ by a factor of $T^{o(1)}$ if necessary, we may assume that the $t'$ are $1$-separated and that the interval $I$ is independent of $t'$, and by passing to a subsequence we may assume that $T = N^{\tau+o(1)}$ for some $1 \leq \tau \leq 1/\delta_1$, then
    $$ \bigg| \sum_{n \in I} \frac{1}{n^{it'}} \bigg| \gg N^{\sigma-\delta_2/\delta_1+o(1)}$$
    for all $t'$.  If we let $\Sigma'$ denote the set of such $t'$, then by Definition \ref{lve-def} we then have (for $\delta_2$ small enough) we have
    $$ E_1(\Sigma') \ll N^{\LV^*_\zeta(\sigma,\tau) + \eps + o(1)} \ll T^{\LV^*_\zeta(\sigma,\tau)/\tau + \eps + o(1)}.$$
    Let $\Sigma$ denote the set of imaginary parts of zeroes under consideration (after the passage to subsequences indicated previously).  By construction, each element of $\Sigma$ lies within $T^{o(1)}$ of an element of $\Sigma'$, and by the Riemann--von Mangoldt formula there are at most $T^{o(1)}$ elements of $\Sigma$ in any unit interval.  Standard additive energy calculations (e.g., using \eqref{energy-asymp}) then gives the bound
    $$ E_1(\Sigma) \ll T^{\LV^*_\zeta(\sigma,\tau)/\tau + \eps + o(1)}.$$
    and the claim follows in this case from \eqref{lvze-bound}.

    The Type II case similarly follows from \eqref{lve-bound} exactly as in the proof of Lemma \ref{zero-from-large}.
\end{proof}

By repeating the proof of \Cref{zero-large-cor-0}, we obtain

\begin{corollary}\label{zeroe-large-cor-0} Let $1/2 < \sigma < 1$ and $\tau_0 > 0$ be fixed.  Then
    $$ \A^*(\sigma)(1-\sigma) \leq \max \left(\sup_{2 \leq \tau < \tau_0} \LV^*_\zeta(\sigma,\tau)/\tau, \sup_{\tau_0 \leq \tau \leq 2\tau_0} \LV^*(\sigma,\tau)/\tau\right).$$
    \end{corollary}

Thus, to obtain upper bounds on $\A^*(\sigma)$, it suffices to control the quantities $\LV^*_\zeta(\sigma,\tau)$, $\LV^*(\sigma,\tau)$, and by \eqref{lvs-energy}, \eqref{lvsz-energy}, it suffices to control the large value energy regions $\Energy, \Energy_\zeta$.

In analogy with \Cref{power-lemma}, we would expect the region $\Energy$ to be closed under the operation $(\sigma,\tau,\rho,\rho^*,s) \mapsto (\sigma,\tau/k, \rho/k, \rho^*/k, s/k)$.  Due to some technical issues (having to do with the fact that the additive energy or double zeta sum may unexpectedly drop when passing to a subsequence), we were not able to prove this claim as stated; however we have the following serviceable substitute.

\begin{lemma}[Raising to a power]\label{power-energy}  If $(\sigma,\tau,\rho,\rho^*,s) \in \Energy$, and $k \geq 1$, then amongst all tuples $(\sigma,\tau/k,\rho',(\rho')^*,s') \in \Energy$ with $\rho' \leq \rho/k$, $(\rho')^* \leq \rho^*/k$, and $s' \leq s/k$, there exists a tuple with $\rho' = \rho/k$, there exists a tuple with $(\rho')^* = \rho^*/k$, and there exists a tuple with $s' = s/k$. (These may be three different tuples.)
\end{lemma}

\begin{proof} By definition, there exists a large value pattern $(N,T,V,(a_n)_{n \in [N,2N]},J,W)$ with $N > 1$ unbounded, $T = N^{\tau+o(1)}$, $V = N^{\sigma+o(1)}$, $|W| = N^{\rho+o(1)}$, $E_1(W) = N^{\rho^*+o(1)}$, and $S(N,W) = N^{s+o(1)}$.  Observe that
    $$ \left( \sum_{n \in [N,2N]} a_n n^{-it} \right)^k = \sum_{n \in [N^k, 2^k N^k]} b_n n^{-it}$$
for some coefficients $b_n = O(n^{o(1)})$.  In particular, partitioning $[N^k, 2^k N^k]$ into $O(1)$ sub-intervals $[N',2N']$ with $N' = N^{k+o(1)}$, we can partition $W$ into $O(1)$ subcollections $W_{N'}$, such that
$$ \left| \sum_{n \in [N',2N']} b_n n^{-it} \right| \gg V^k = (N')^{\sigma+o(1)}$$
for all $t \in W_{N'}$.  Again by the pigeonhole principle, one of the $W_{N'}$ must have cardinality $N^{\rho+o(1)}$, one must have energy $N^{\rho^*+o(1)}$, and one must have double zeta sum $N^{s+o(1)}$ (but these may be different $W_{N'}$).  Each of these $W_{N'}$ then give the different conclusions to the lemma.
\end{proof}

Now we turn to the estimation of additive energy exponent $\A^*(\sigma)$.  Besides the trivial upper bound $\A^*(\sigma) \leq 3 \A(\sigma)$, the best upper bounds for $\A^*(\sigma)$ in the literature remain that of Heath--Brown \cite{heathbrown_zero_1979}, who showed for any fixed $1/2 \leq \sigma \leq 1$ that $\A^*(\sigma)$ could be bounded by
\begin{align*}
        \frac{10-11\sigma}{(2-\sigma)(1-\sigma)} & \hbox{ for } 1/2 \leq \sigma \leq 2/3;\\
        \frac{18-19\sigma}{(4-2\sigma)(1-\sigma)} & \hbox{ for } 2/3 \leq \sigma \leq 3/4;\\
        \frac{12}{4\sigma-1} & \hbox{ for } 3/4 \leq \sigma \leq 1.
\end{align*}
A key input in the analysis was the following relation between additive energy and size of a large value pattern:

\begin{theorem}[Heath-Brown relation]\label{hbt}\cite{heathbrown_zero_1979} If $(\sigma,\tau,\rho,\rho^*,s) \in \Energy$, then one has
    $$ \rho^* \leq 1-2\sigma + \frac{1}{2}\max\left(\rho+1, 2\rho, \frac{5}{4}\rho+\frac{\tau}{2}\right) + \frac{1}{2}\max\left(\rho^*+1, 4\rho, \frac{3}{4}\rho^*+\rho+\frac{\tau}{2}\right).$$
    \end{theorem}

\begin{proof} If $(\sigma,\tau,\rho,\rho^*,s)\in\Energy$ then by \Cref{lv-edef}, there exists a large value pattern $(N,T,V,(a_n)_{n \in [N,2N]},J,W)$ with $N>1$ unbounded, $V = N^{\sigma+o(1)}$, $T = N^{\tau+o(1)}$, $V = N^{\sigma+o(1)}$, $|W| = N^{\rho+o(1)}$, $E_1(W) = N^{\rho^*+o(1)}$ and $S(N,W) = N^{s+o(1)}$.  Applying \cite[(33)]{heathbrown_zero_1979} (and adjusting the notation appropriately), we conclude that
\begin{align*}
E_1(W) \ll T^\eps N^{1-2\sigma} &(|W|N + |W|^2 + |W|^{5/4} T^{1/2})^{1/2} \\
&\times (E_1(W)N + |W|^4 + E_1(W)^{3/4} N T^{1/2})^{1/2}.
\end{align*}
Comparing exponents, we obtain the claim.
\end{proof}

Given the work on large value estimates since the time of Heath-Brown's paper, it is not surprising that we are able to improve on these bounds in several ranges.  Our main result (\Cref{Add-est} below) is obtained using an optimisation routine that may be viewed as a higher-dimensional analogue of the method used to derive \Cref{beta-table}. Each theorem in the literature relating the quantities $\sigma$, $\tau$, $\rho$, $\rho^*$ and $s$, such as \Cref{hbt} or any large value estimate, is represented as a five-dimensional polytope containing feasible $(\sigma, \tau, \rho, \rho^*, s)$ tuples. The polytopes are combined by taking their intersection, then projected onto the $(\sigma, \tau, \rho^*)$ dimensions. For each fixed $\sigma$, an upper bound of $A^*(\sigma)(1-\sigma)$ may be calculated as $\sup \rho^*/\tau$ where the supremum is taken over all tuples in the polytope projection (with $\sigma$ fixed). 

This process generates a machine ``proof", which, while not certifiably correct due to issues of code fidelity, provides a blueprint for the rough steps required to arrive at a formal proof. The machine-generated proofs are then transcribed manually into human-readable proofs which are available in the ANTEDB. 

\begin{theorem}[New additive energy estimates]\label{Add-est}  Let $7/10 \leq \sigma < 1$ be fixed.
    \begin{itemize}
    \item[(i)] For $3/4 \le \sigma \le 5/6$ one has
    \[
    \A^*(\sigma)(1-\sigma) \le \max\left(\frac{18 - 19\sigma}{2(3\sigma - 1)}, \frac{4(10 - 9\sigma)}{5(4\sigma - 1)}\right).
    \]
    \item[(ii)] For $7/10 \le \sigma \le 3/4$, one has
    \[
    \A^*(\sigma)(1-\sigma) \le \max\left(\frac{5(18 - 19\sigma)}{2(5\sigma + 3)}, \frac{2(45 - 44\sigma)}{2\sigma + 15}\right).
    \]
    \item[(iii)] For $173/229 \le \sigma \le 443/586$, one has 
    \[
    \A^*(\sigma)(1-\sigma) \le \max\left(\frac{173 - 270\sigma}{16(93 - 125\sigma)}, \frac{653 - 890\sigma}{10(93 - 125\sigma)}, \frac{1151 - 1190\sigma}{20(15\sigma - 2)}\right).
    \]
    \item[(iv)] For $443/586 \le \sigma \le 373/493$, one has
    \[
    \A^*(\sigma)(1-\sigma) \le \max\left(\frac{593 - 810\sigma}{5(171 - 230\sigma)}, \frac{4(266 - 275\sigma)}{5(55\sigma - 7)}\right).
    \]
    \item[(v)] For $373/493 \le \sigma \le 103/136$, one has
    \[
    \A^*(\sigma)(1-\sigma) \le \max\left(\frac{533 - 730\sigma}{30(26 - 35\sigma)}, \frac{3(26 - 33\sigma)}{85\sigma - 62}, \frac{174 - 185\sigma}{31\sigma + 2}\right).
    \]
    \item[(vi)] For $103/136 \le \sigma \le 42/55$, one has
    \[
    \A^*(\sigma)(1-\sigma) \le \max\left(\frac{72 - 91\sigma}{7(11\sigma - 8)}, \frac{5(18 - 19\sigma)}{2(5\sigma + 3)}\right).
    \]
    \item[(vii)] For $42/55 \le \sigma \le 79/103$, one has
    \[
    \A^*(\sigma)(1-\sigma) \le \max\left(\frac{18 - 19\sigma}{6(15\sigma - 11)}, \frac{3(18-19\sigma)}{4(4\sigma-1)}\right).
    \]
    \item[(viii)] For $79/103 \le \sigma \le 84/109$, one has
    \[
    \A^*(\sigma)(1-\sigma) \le \max\left(\frac{18 - 19\sigma}{2(37\sigma - 27)}, \frac{5(18 - 19\sigma)}{2(13\sigma - 3)}\right).
    \]
    \item[(ix)] For $84/109 \le \sigma \le 5/6$, one has
    \[
    \A^*(\sigma)(1-\sigma) \le \max\left(\frac{18 - 19\sigma}{9(3\sigma - 2)}, \frac{4(10 - 9\sigma)}{5(4\sigma - 1)}\right).
    \]
    \end{itemize}
\end{theorem}

\begin{proof}  We just describe the proof of (i) here; the proofs of the remaining claims use similar methods (as well as inputs such as the Guth--Maynard estimates in \cite{guth-maynard}) and can be found in the ANTEDB.

Throughout assume that $3/4 \le \sigma \le 5/6$. Choose
\[
\tau_0 = 8\sigma - 4.
\]
We will show that
\begin{equation}\label{imphb-lver-ineq}
\rho^*/\tau \le \begin{cases}
\dfrac{18 - 19\sigma}{2(3\sigma - 1)},&3/4 \le \sigma < 4/5,\\
\dfrac{7(1 - \sigma)}{3\sigma - 1},&4/5 \le \sigma \le 5/6,
\end{cases}
\end{equation}
for all $(\sigma, \tau, \rho, \rho^*, s) \in \Energy$ for which $\tau_0 \le \tau \le 2\tau_0$, and that
\begin{equation}\label{imphb-zlver-ineq}
\rho^*/\tau \le \max\left(\frac{45 - 46\sigma}{4(4\sigma - 1)}, \frac{4(10 - 9\sigma)}{5(4\sigma - 1)}\right),
\end{equation}
for all $(\sigma, \tau, \rho, \rho^*, s) \in \Energy_\zeta$ such that $2 \le \tau \le \tau_0$. The desired result (i) then follows from Corollary \ref{zeroe-large-cor-0} and computing the piecewise maximum of \eqref{imphb-lver-ineq} and \eqref{imphb-zlver-ineq}.

First, consider \eqref{imphb-zlver-ineq}. Suppose that $(\sigma, \tau, \rho, \rho^*, s)\in \Energy_\zeta$ with $3/4 \le \sigma \le 5/6$ and $2 \le \tau \le \tau_0$. Then, from \eqref{twelfth-bound} we have
\begin{equation}\label{hb-lv-rho-form}
\rho \le 2\tau - 12(\sigma - 1/2).
\end{equation}
Furthermore, by \eqref{huxley-lvt} and \Cref{power-lemma} with $k = 2$, one has $\rho \le 2\max(2 - 2\sigma, 4 - 6\sigma + \tau/2)$. However since $\tau \le \tau_0 = 8\sigma - 4$, this simplifies to
\begin{equation}
\label{huxley-lv-rho-form2}
\rho \le 4 - 4\sigma.
\end{equation}
Since $\sigma \ge 3/4$, this also implies that $\rho \le 1$. For future reference we also note that
\begin{equation}\label{zlver:tau-gradient-1}
\frac{4}{5} < \max\left(\frac{45 - 46\sigma}{4(4\sigma - 1)}, \frac{4(10 - 9\sigma)}{5(4\sigma - 1)}\right) < 2,\qquad (3/4 \le \sigma \le 5/6).
\end{equation}
By \Cref{hbt}, one has
\[
\rho^* \leq 1-2\sigma + \frac{1}{2}\max\left(\rho+1, 2\rho, \frac{5}{4}\rho+\frac{\tau}{2}\right) + \frac{1}{2}\max\left(\rho^*+1, 4\rho, \frac{3}{4}\rho^*+\rho+\frac{\tau}{2}\right).
\]
Since $\rho \le 1$, one has $\rho + 1 \ge 2\rho$. Thus the middle term in the first maximum may be omitted, and we are left with two cases to consider.

\textbf{Case 1:} If $\rho + 1 \ge 5\rho/4 + \tau/2$ then
\[
\rho^* \le 1 -2\sigma + \frac{\rho + 1}{2} + \frac{1}{2}\max\left(\rho^* + 1, 4\rho, \frac{3}{4}\rho^*+\rho +\frac{\tau}{2}\right).
\]
Solving for $\rho^*$ gives
\[
\rho^* \le \max\left(4 - 4\sigma + \rho, \frac{3}{2} - 2\sigma + \frac{5}{2}\rho, \frac{2}{5}(6 - 8\sigma + \tau + 4\rho)\right).
\]
Applying \eqref{huxley-lv-rho-form2} to each term,
\begin{align*}
\rho^* &\le \max\left(8 - 8\sigma, \frac{23}{2} - 12\sigma, \frac{2}{5}(22 - 24\sigma + \tau)\right) \\
&\le \max\left(\frac{45 - 46\sigma}{4(4\sigma - 1)}\tau, \frac{4(10 - 9\sigma)}{5(4\sigma - 1)}\tau\right),
\end{align*}
i.e.\ \eqref{imphb-zlver-ineq} holds. The last inequality may be verified by inspecting the growth rates with respect to $\tau$ of each term (using \eqref{zlver:tau-gradient-1}), and checking that the desired inequality holds at $\tau = 2$.

\textbf{Case 2:} If $\rho + 1 < 5\rho/4 + \tau/2$, then
\[
\rho^* \le 1 - 2\sigma + \frac{5}{8}\rho + \frac{\tau}{4} + \frac{1}{2}\max\left(\rho^* + 1, 4\rho, \frac{3}{4}\rho^*+\rho +\frac{\tau}{2}\right).
\]
Solving for $\rho$ gives
\[
\rho^* \le \max\left(3 - 4\sigma + \frac{\tau}{2} + \frac{5}{4}\rho, 1 - 2\sigma + \frac{\tau}{4} + \frac{21}{8}\rho, \frac{8 - 16\sigma + 4\tau + 9\rho}{5}\right).
\]
If $\tau \ge 4\sigma - 1$, then apply \eqref{huxley-lv-rho-form2} termwise to get
\begin{align*}
\rho^* &\le \max\left(8 - 9\sigma + \frac{\tau}{2}, \frac{23}{2} - \frac{25}{2} + \frac{\tau}{4}, \frac{4}{5}(11 - 13 \sigma + \tau)\right)\\
&\le \max\left(\frac{45 - 46\sigma}{4(4\sigma - 1)}\tau, \frac{4(10 - 9\sigma)}{5(4\sigma - 1)}\tau\right),
\end{align*}
since the RHS is increasing faster in $\tau$ and at $\tau = 4\sigma - 1$ we have equality.

On the other hand if $\tau < 4\sigma - 1$ then we apply \eqref{hb-lv-rho-form} termwise to get
\begin{align*}
    \rho^* &\le \max\left(\frac{21}{2} - 19\sigma + 3\tau, \frac{67 -134\sigma + 22\tau}{4}, \frac{2}{5}(31 - 62\sigma + 11\tau)\right)\\
    &\le \max\left(\frac{45 - 46\sigma}{4(4\sigma - 1)}\tau, \frac{4(10 - 9\sigma)}{5(4\sigma - 1)}\tau\right),
\end{align*}
by inspecting growth rates in $\tau$ and noting that at $\tau = 4\sigma - 1$ one has equality. 

Thus we have shown that if $(\sigma, \tau, \rho, \rho^*, s) \in \Energy_\zeta$ with $3/4 \le \sigma \le 5/6$ and $2 \le \tau \le 8\sigma - 4$, then
\[
\rho^*/\tau \le \min\left(\frac{45 - 46\sigma}{4(4\sigma - 1)}, \frac{4(10 - 9\sigma)}{5(4\sigma - 1)}\right),
\]
which is \eqref{imphb-zlver-ineq}.

Now consider \eqref{imphb-lver-ineq}. Suppose that $\tau_0 \le \tau \le 2\tau_0$ and $(\sigma, \tau, \rho, \rho^*, s) \in \Energy$. Note that the interval $[\tau_0, 2\tau_0]$ is covered by intervals $I_k := [(4\sigma - 2)k, (4\sigma - 2)(k + 1)]$ with $k = 2, 3$. Suppose that $\tau \in I_k$, and write 
\[
\tau' := \tau/k.
\]
Then, by \eqref{huxley-lvt} and $\tau' \ge 4\sigma - 2$ one has
\[
\rho/k \le \max(2 - 2\sigma, 4 - 6\sigma + \tau') = 4 - 6\sigma + \tau'.
\]
Also, from \eqref{huxley-lvt} and $\tau \le (4\sigma - 2)(k + 1)$ one has
\[
\rho/(k + 1) \le \max(2 - 2\sigma, 4 - 6\sigma + \tau/(k + 1)) \le 2 - 2\sigma
\]
so that for $k = 2,3$ one has $\rho/k \le (2 - 2\sigma)(k + 1)/k \le 3 - 3\sigma$.
In summary,
\begin{equation}\label{ze-ihb-rho-bound-k}
\rho/k \le \min(3 - 3\sigma, 4 - 6\sigma + \tau').
\end{equation}
Next, by \Cref{power-energy},
\[
(\sigma, \tau', \rho', \rho^*/k, s') \in \Energy
\]
for $\tau' := \tau/k$ and some $\rho' \le \rho/k$ and $s' \le s/k$.

Applying \Cref{hbt} to this tuple, then applying $\rho' \le \rho/k$, one has
\begin{align*}
\frac{\rho^*}{k} \leq 1-2\sigma &+ \frac{1}{2}\max\left(\frac{\rho}{k}+1, \frac{2\rho}{k}, \frac{5\rho}{4k} + \frac{\tau'}{2}\right) \\
&+ \frac{1}{2}\max\left(\frac{\rho^*}{k}+1, \frac{4\rho}{k}, \frac{3}{4}\frac{\rho^*}{k} +\frac{\rho}{k}+\frac{\tau'}{2}\right).
\end{align*}
By \eqref{ze-ihb-rho-bound-k} one has $\rho/k \le 1$ (since $\sigma \ge 3/4$) so there are only two cases to consider:

\textbf{Case 1:} $\rho/k + 1 \ge 5\rho/(4k) + \tau'/2$ then
\[
\frac{\rho^*}{k} \le \frac{3}{2} - 2\sigma + \frac{\rho}{2k} + \frac{1}{2}\max\left(\frac{\rho^*}{k}+1, \frac{4\rho}{k}, \frac{3}{4}\frac{\rho^*}{k} +\frac{\rho}{k}+\frac{\tau'}{2}\right).
\]
Solving for $\rho^*/k$, we get
\[
\frac{\rho^*}{k} \le \max\left(4(1 - \sigma) + \frac{\rho}{k}, \frac{(3 - 4\sigma) + 5\rho/k}{2}, \frac{2}{5}((6 - 8\sigma) + \tau' + \frac{4\rho}{k})\right).
\]
If $\tau' \ge 3\sigma - 1$ then \eqref{ze-ihb-rho-bound-k} reduces to $\rho/k \le 3-3\sigma$. Substituting this bound gives
\[
\rho^*/k \le \max(7 - 7\sigma, 9 - 19\sigma/2, 32(1 - \sigma)/5).
\]
However the RHS is bounded by 
\begin{equation}\label{ze-ihb-rho-1}
\frac{18 - 19\sigma}{6\sigma - 2}\tau'
\end{equation}
for all $\tau' \ge 3\sigma - 1$ since the desired bound holds at $\tau' = 3\sigma - 1$ (where one has equality).

On the other hand, if $4\sigma - 2 \le \tau' \le 3\sigma - 1$ then \eqref{ze-ihb-rho-bound-k} reduces to $\rho/k \le 4 - 6\sigma + \tau'$. Substituting this bound gives
\begin{equation}
\begin{split}
\rho^*/k &\le \max\left(8 -10\sigma + \tau',\frac{23 - 34\sigma + 5\tau'}{2}, \frac{2}{5}(22 - 32\sigma +  5\tau')\right)\\
&\le \frac{18 - 19\sigma}{6\sigma - 2}\tau'\label{ze-ihb-rho-2}
\end{split}
\end{equation}
where the last inequality may be established by checking that it holds at both $\tau' = 4\sigma - 2$ and $\tau' = 3\sigma - 1$ (where one has equality). To summarise, by taking $k = 2, 3$ in the \eqref{ze-ihb-rho-1} and \eqref{ze-ihb-rho-2}, one has 
\[
\rho^* \le \frac{18 - 19\sigma}{6\sigma - 2}\tau,\qquad (\tau_0 \le \tau \le 2\tau_0)
\]
in this case. 

\textbf{Case 2:} $\rho/k + 1 < 5\rho/(4k) + \tau'/2$ then
\[
\rho^*/k \leq 1-2\sigma + \frac{5\rho}{8k} + \frac{\tau'}{4} + \frac{1}{2}\max\left(\frac{\rho^*}{k}+1, \frac{4\rho}{k}, \frac{3}{4}\frac{\rho^*}{k} + \frac{\rho}{k}+\frac{\tau'}{2}\right).
\]
Solving for $\rho^*/k$ gives
\[
\rho^*/k \le \max\left(3 - 4\sigma + \frac{\tau'}{2} + \frac{5}{4}\frac{\rho}{k}, 1 - 2\sigma + \frac{\tau'}{4} + \frac{21}{8}\frac{\rho}{k}, \frac{1}{5}(8 - 16\sigma + 4\tau' + 9\frac{\rho}{k})\right).
\]
Proceeding as before, if $\tau' \ge 3\sigma - 1$ then \eqref{ze-ihb-rho-bound-k} becomes $\rho/k \le 3 - 3\sigma$, and substituting gives
\begin{align*}
\rho^*/k \le \max\left(\frac{27-31\sigma}{4} + \frac{\tau'}{2}, \frac{71-79\sigma}{8}+\frac{\tau'}{4}, 7-\frac{43}{5}\sigma + \frac{4}{5}\tau'\right).
\end{align*}
One may verify that the RHS is bounded by 
\[
\frac{18 - 19\sigma}{6\sigma - 2}\tau'
\]
(with some room to spare) by checking at the endpoint $\tau' = 3\sigma - 1$. 

Similarly, if $4\sigma - 2 \le \tau' \le 3\sigma - 1$ then using $\rho/k \le 4 - 6\sigma + \tau'$ from \eqref{ze-ihb-rho-bound-k} one has
\[
\rho^*/k \le \max\left(8 - \frac{23\sigma}{2} + \frac{7\tau'}{4}, \frac{23}{2} - \frac{71\sigma}{4} + \frac{23\tau'}{8}, \frac{44 - 70\sigma + 13\tau'}{5}\right).
\]
One can check that the RHS is bounded by 
\[
\frac{18 - 19\sigma}{6\sigma - 2}\tau'
\]
by checking the required inequalities hold at $\tau' = 4\sigma - 2$ and $\tau' = 3\sigma - 1$ (in each case, with some room to spare). 

Combining all the cases, by taking $k= 2,3$ we have shown that for $3/4 \le \sigma \le 4/5$ and $\tau_0 \le \tau \le 2\tau_0$ one has
\[
\rho^* \le \frac{18 - 19\sigma}{6\sigma - 2}\tau
\]
which is the first part of \eqref{imphb-lver-ineq}.

The case for $4/5 \le \sigma \le 5/6$ may be treated similarly. Here one needs to verify that
\[
\max(7 - 7\sigma, 9 - 19\sigma/2, 32(1 - \sigma)/5) \le \frac{7(1 - \sigma)}{3\sigma - 1}\tau',
\]
\[
\max\left(\frac{27-31\sigma}{4} + \frac{\tau'}{2}, \frac{71-79\sigma}{8}+\frac{\tau'}{4}, 7-\frac{43}{5}\sigma + \frac{4}{5}\tau'\right) \le \frac{7(1 - \sigma)}{3\sigma - 1}\tau'
\]
for $\tau' \ge 3\sigma - 1$, and that
\[
\max\left(8 -10\sigma + \tau',\frac{23 - 34\sigma + 5\tau'}{2}, \frac{2}{5}(22 - 32\sigma +  5\tau')\right) \le \frac{7(1 - \sigma)}{3\sigma - 1}\tau',
\]
\[
\max\left(8 - \frac{23\sigma}{2} + \frac{7\tau'}{4}, \frac{23}{2} - \frac{71\sigma}{4} + \frac{23\tau'}{8}, \frac{44 - 70\sigma + 13\tau'}{5}\right) \le \frac{7(1 - \sigma)}{3\sigma - 1}\tau'
\]
for $4\sigma - 2 \le \tau' \le 3\sigma - 1$. The treatment is analogous to before, so we omit the proof.
\end{proof}

Further additive energy estimates are recorded in the ANTEDB. \Cref{fig:additive-energy-estimate} displays currently known upper bounds on $\A^*(\sigma)$.

\begin{figure}
\centering
\includegraphics[width=10cm]{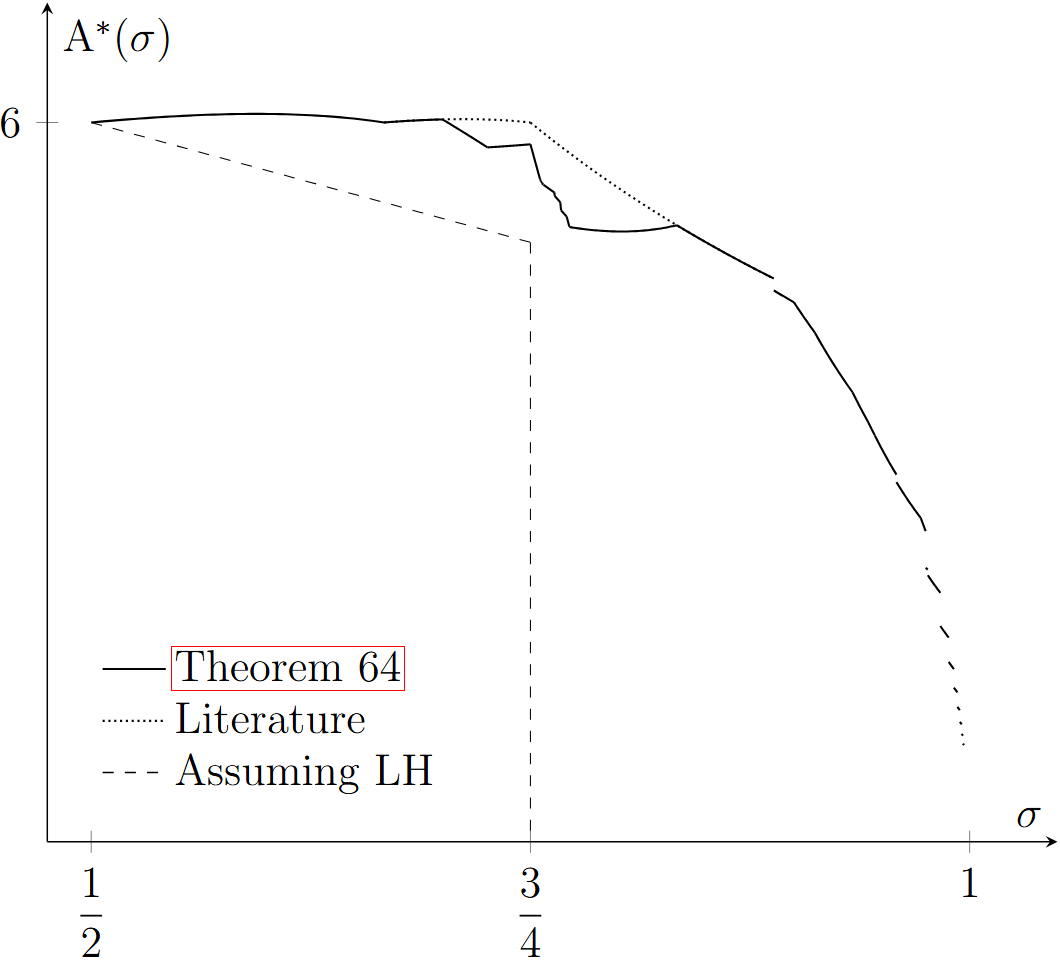}
\caption{Plot of bounds on $\A^*(\sigma)$ in \Cref{Add-est}, against existing literature bounds.}
\label{fig:additive-energy-estimate}
\end{figure}

%\subsection{Acknowledgments}

%\bibliographystyle{alpha}
%\bibliography{references}

\end{document}